\documentclass{amsart}
\usepackage{amssymb,amsmath,amsthm,amsfonts}
\usepackage {latexsym}
\usepackage{bbm}

\allowdisplaybreaks
\newcommand{\nc}{\newcommand}
\nc{\les}{\lesssim}
\nc{\nit}{\noindent}
\nc{\nn}{\nonumber}
\nc{\D}{\partial}
\nc{\diff}[2]{\frac{d #1}{d #2}}
\nc{\diffn}[3]{\frac{d^{#3} #1}{d {#2}^{#3}}}
\nc{\pdiff}[2]{\frac{\partial #1}{\partial #2}}
\nc{\pdiffn}[3]{\frac{\partial^{#3} #1}{\partial{#2}^{#3}}}
\nc{\abs}[1] {\lvert #1 \rvert}
\nc{\cAc}{{\cal A}_c}
\nc{\cE}{{\cal E}}
\nc{\cF}{{\cal F}}
\nc{\cP}{{\cal P}}
\nc{\cV}{{\cal V}}
\nc{\cQ}{{\cal Q}}
\nc{\cGin}{{\cal G}_{\rm in}}
\nc{\cGout}{{\cal G}_{\rm out}}
\nc{\cO}{{\cal O}}
\nc{\Lav}{{\cal L}_{\rm av}}
\nc{\cL}{{\cal L}}
\nc{\cB}{{\cal B}}
\nc{\cZ}{{\cal Z}}
\nc{\cR}{{\cal R}}
\nc{\cT}{{\cal T}}
\nc{\cY}{{\cal Y}}
\nc{\cX}{{\cal X}}
\nc{\cXT}{{{\cal X}(T)}}
\nc{\cBT}{{{\cal B}(T)}}
\nc{\vD}{{\vec \mathcal{D}}}
\nc{\efield}{\mathcal{E}}
\nc{\vE}{{\vec \efield}}
\nc{\vB}{{\vec \mathcal{B}}}
\nc{\vH}{{\vec \mathcal{H}}}
\nc{\ty}{{\tilde y}}
\nc{\tu}{{\tilde u}}
\nc{\tV}{{\tilde V}}
\nc{\Pc}{{\bf P_c}}
\nc{\bx}{{\bf x}}
\nc{\bX}{{\bf X}}
\nc{\bXYZ}{{\bf XYZ}}
\nc{\bY}{{\bf Y}}
\nc{\bF}{{\bf F}}
\nc{\bS}{{\bf S}}
\nc{\dV}{{\delta V}}
\nc{\dE}{{\delta E}}
\nc{\TT}{{\Theta}}
\nc{\dPsi}{{\delta\Psi}}
\nc{\order}{{\cal O}}
\nc{\Rout}{R_{\rm out}}
\nc{\eplus}{e_+}
\nc{\eminus}{e_-}
\nc{\epm}{e_\pm}
\nc{\eps}{\varepsilon}
\nc{\vnabla}{{\vec\nabla}}
\nc{\G}{\Gamma}
\nc{\w}{\omega}
\nc{\mh}{h}
\nc{\mg}{g}
\nc{\vphi}{\varphi}
\nc{\tlambda}{\tilde\lambda}
\nc{\be}{\begin{equation}}
\nc{\ee}{\end{equation}}
\nc{\ba}{\begin{eqnarray}}
\nc{\ea}{\end{eqnarray}}

\nc{\g}{\gamma}
\nc{\ol}{\overline}

\newtheorem{theorem}{Theorem}[section]
\newtheorem{lemma}[theorem]{Lemma}
\newtheorem{prop}[theorem]{Proposition}
\newtheorem{corollary}[theorem]{Corollary}
\newtheorem{defin}[theorem]{Definition}

\nc{\pT}{\partial_T}
\nc{\pz}{\partial_z}
\nc{\pt}{\partial_t}
\nc{\la}{\langle}
\nc{\ra}{\rangle}
\nc{\infint}{\int_{-\infty}^{\infty}}
\nc{\halfwidth}{6.5cm}
\nc{\figwidth}{10cm}
\newcommand{\f}{\frac}

\nc{\nlayers}{L} \nc{\nsectors}{M}
\nc{\indicator}{\mathbf{1}}
\nc{\Rhole}{R_{\rm hole}}
\nc{\Rring}{R_{\rm ring}}
\nc{\neff}{n_{\rm eff}}
\nc{\Frem}{F_{\rm rem}}
\nc{\R}{\mathbb R}
\nc{\Z}{\mathbb Z}
\nc{\DD}{\Delta}
\nc{\cD}{\mathcal D}
\nc{\lnorm}{\left\|}
\nc{\rnorm}{\right\|}
\nc{\rnormp}{\right\|_{\ell^{p,\eps}}}
\nc{\rar}{\rightarrow}
\begin{document}

\begin{abstract}

	We investigate $L^1(\R^2)\to L^\infty(\R^2)$ dispersive estimates for the   Schr\"odinger operator $H=-\Delta+V$ when there are obstructions, resonances or an eigenvalue, at zero energy. In particular, we show that the existence of an  s-wave resonance at zero energy does not destroy the $t^{-1}$ decay rate.  We also show that if there is a p-wave resonance or an eigenvalue at zero energy then there is a time dependent operator $F_t$ satisfying $\|F_t\|_{L^1\to L^\infty} \lesssim 1$  such that
$$\|e^{itH}P_{ac}-F_t\|_{L^1\to L^\infty} \lesssim |t|^{-1},\,\,\,\,\,\text{ for } |t|>1.$$
We also establish a weighted dispersive estimate with $t^{-1}$ decay rate in the case when there is an eigenvalue at zero energy but no resonances.
\end{abstract}

\title[Dispersive Estimates for  the Schr\"odinger Equation]{\textit{Dispersive estimates for
Schr\"{o}dinger operators in dimension two with obstructions at zero energy}}

\author[M.~B. Erdo\smash{\u{g}}an, W.~R. Green]{M. Burak Erdo\smash{\u{g}}an and William~R. Green}
\thanks{The first author was partially supported by NSF grant  DMS-0900865.}
\address{Department of Mathematics \\
University of Illinois \\
Urbana, IL 61801, U.S.A.}
\email{berdogan@math.uiuc.edu}
\address{Department of Mathematics and Computer Science\\
Eastern Illinois University \\
Charleston, IL 61920, U.S.A.}
\email{wrgreen2@eiu.edu}
%\subjclass{35Q41, 42B20}

\maketitle
\section{Introduction}
Consider the Schr\"{o}dinger operator  $H=-\Delta+V$ in $\R^2$, where $V$ is a real-valued potential.
Let $P_{ac}$ be the
orthogonal projection onto the absolutely continuous subspace of $L^2(\R^2)$, which is determined by $H$.
In \cite{Sc2}, Schlag proved that
$$\|e^{itH}P_{ac}\|_{L^1(\R^2)\to L^\infty(\R^2)}\lesssim |t|^{-1}$$
under the decay assumption $|V|\lesssim \langle x\rangle^{-3-}$ and the assumption that zero is neither an eigenvalue nor a resonance of~$H$.

Recall that (see, e.g., \cite{JN} or Section~\ref{sec:spectral} below) there is a resonance at zero energy if there is a distributional solution to the equation $H\psi =0$
where $\psi\notin L^2(\R^2)$ but $\psi\in L^p(\R^2)$ for some $p\in (2,\infty]$. There are two possible cases, either $\psi\in L^\infty (\R^2)$ and $\psi\notin L^p(\R^2)$ for any  $p<\infty$ or $\psi\in L^p(\R^2)$ for all
$p\in (2,\infty]$. In the case of $\psi\in L^\infty(\R^2)$ only,
the resonance is called an s-wave resonance.  In the second case, we
say there is a p-wave resonance.
We say that there is an eigenvalue at zero if $\psi\in L^2(\R^2)$. This definition for resonances differs from the case of dimension $n=3$ in which $\psi $ lies in weighted $L^2$
spaces.

We note that in the case of $V\equiv 0$ the function $\psi \equiv 1$ solves $H\psi =0 $  which corresponds to
an s-wave resonance.  It is important to note that in spite of this obstruction, the free evolution decays in time at the rate $t^{-1}$.

Much is known about dispersive estimates for the Schr\"odinger equation when zero is regular.
The history goes back to
Rauch, \cite{Rauch}, who studied the local decay in dimension three.  In \cite{Rauch}, he noted that
in the generic case, i.e. when there may be eigenvalues or resonances, the evolution decays at a rate of
$|t|^{-1/2}$ as $t\to\infty$ on exponentially weighted
$L^2$ spaces.  In the case when there are no eigenvalues or resonances, it was shown that the
 decay rate is $|t|^{-3/2}$.
  Jensen and Kato, \cite{JenKat}, improved this result to polynomially weighted $L^2$ spaces in dimension three, and
higher dimensions, \cite{Jen,Jen2}.   In \cite{JenKat}, it was noted that the presence of a zero energy eigenvalue or
resonance destroys the $|t|^{-3/2}$ decay even if one projects away from the eigenspace in dimension three.

Local decay estimates in the two dimensional case when zero is regular
were studied by Murata in \cite{Mur}.  Murata was able to prove
an estimate on weighted $L^2$ spaces that decays like $t^{-1}(\log t)^{-2}$, which is integrable at infinity.  Such estimates
have been used in analysis of the stability of certain two-dimensional non-linear equations.

The first result to discuss global decay, $L^1\to L^\infty$ estimates, was due to Journ\'e, Soffer and Sogge in
\cite{JSS}.  Their result relied on the integrability of $t^{-n/2}$ at infinity and is thus restricted to $n\geq 3$.  Much
is now known in this direction, mainly in dimension three.  Rodnianski and Schlag established such estimates in
dimension three, \cite{RodSch}, in addition to establishing Strichartz estimates.  Following from their methods, a great
number of results in dimension three followed, particularly \cite{goldberg, Gol2, GS}.  The one dimensional problem
was studied by Weder, \cite{Wed} and Goldberg and Schlag \cite{GS}.
Also see \cite{ES2,Yaj3, goldE} for global estimates in
the three-dimensional case when there is an eigenvalue
and/or resonance at zero energy, and \cite{ES} for a similar result for the  matrix Schr\"odinger equation.

There have also been studies of the wave-operators in dimension two.  In particular Yajima, \cite{Yaj2} established
that the wave operators are bounded on $L^p(\R^2)$ for $1<p<\infty$ if zero is regular.  The hypotheses
on the potential $V$ were
relaxed slightly in \cite{JenYaj}.  This result would imply global dispersive estimates if extended to the full range of $p$,
$1\leq p\leq \infty$.  High frequency dispersive estimates, similar to those obtained in \cite{Sc2} stated as
Theorem~\ref{high eng thm} below were obtained by Moulin, \cite{Mou}, under an integrability condition on the potential.

In this paper we investigate $L^1 \to L^\infty$ dispersive estimates in $\R^2$ when zero energy is not a regular point of  the spectrum of the operator $H=-\Delta+V$.
Our goal is to prove the following theorem.
\begin{theorem}\label{big thm}

	Assume that $|V(x)|\lesssim \langle x\rangle^{-\beta}$.  If there is only an s-wave resonance
	at zero energy, then for $\beta>4$, we have
	$$\|e^{itH} P_{ac}(H)\|_{L^1\to L^\infty}\lesssim |t|^{-1}.$$
	If there is a p-wave resonance or eigenvalue at zero, then for $\beta>6$, there is a time-dependent operator
	$F_t$ such that
	$$\|e^{itH} P_{ac}(H)-F_t \|_{L^1\to L^\infty}\lesssim |t|^{-1},\,\,\,\,\,\,\,\,\,|t|>1,$$
	with
	$$\sup_t \| F_t \|_{L^1\to L^\infty}\lesssim 1.$$
\end{theorem}

Note that it is natural to have the $t^{-1}$ decay rate in the  case of an s-wave resonance since the free Schr\"odinger has an s-wave resonance at zero energy. The reason that we can not get any decay in the case of a p-wave resonance or the zero eigenvalue is the behavior of the resolvent around zero energy. In the three dimensional case the resolvent $(H-z^2)^{-1}$ has an expansion of the form
$$(H-z^2)^{-1}=-G_{-2}z^{-2}+G_{-1}z^{-1}+O(1),\,\,\,z\to 0, \,\,\Im(z)>0. $$
The most singular term $G_{-2}z^{-2}$ gives the Riesz projection to zero energy eigenspace.
If one projects away from the zero  eigenspace, the worst singularity is $\frac1z$, which allows for $|t|^{-1/2}$
decay as  $t\to\pm \infty$, see \cite{ES2}. However, in the two dimensional case the resolvent expansion around zero
contains logarithmic terms. In particular, in the general case of zero energy resonances (even if one projects away the zero energy eigenspace), the most singular term is of the form $\frac{1}{z^2 \log(z)}$, which does not allow for any polynomial decay in $t$.
It may be possible to get a decay of the form $\frac1{\log(t)}$ as in \cite{Mur} but we won't pursue this issue here. However, it  is possible to improve this theorem in the case when zero is an eigenvalue but there are no resonances at zero. In particular, we show the following.

\begin{theorem}\label{eval only thm}

	Assume that $|V(x)|\les \la x\ra^{-\beta}$ for some $\beta>11$.
	If zero is an eigenvalue of $H=-\Delta+V$ and there are neither  s-wave nor
	p-wave resonances at
	zero,  then
	\begin{align*}
		\|e^{itH}P_{ac}\|_{L^{1,1+}\to L^{\infty,-1-}}\les |t|^{-1}.
	\end{align*}

\end{theorem}
Here $L^{1,1+}$ is the weighted $L^1$ space defined by
$L^{1,1+}(\R^2):=\{f : \int_{\R^2} |f(x)| \la x\ra^{1+}\, dx<\infty\}$.  Similarly,
$L^{\infty,-1-}=\{f: \la x\ra^{-1-}f\in L^\infty\}$.

Let   $\chi$ is an even smooth function supported in $[-\lambda_1,\lambda_1]$ and $\chi(x)=1$ for $|x|<\lambda_1/2$.
Let $K_{\lambda_1}$ be the kernel of $e^{itH}\chi(H)P_{ac}$:
\be\label{klambda}
K_{\lambda_1}(x,y)=\f{1}{\pi i}\int_0^\infty e^{it\lambda^2}\lambda \chi(\lambda)  [R_V^+(\lambda^2)-R_V^-(\lambda^2)](x,y)
d\lambda,
\ee
where
$$
R_V^\pm(\lambda^2)= R_V(\lambda^2\pm i0)=(H-(\lambda^2 \pm i0))^{-1}
$$
is the perturbed resolvent. By the limiting absorption principle, these boundary values are bounded operators on
weighted $L^2$-spaces, see e.g.~\cite{agmon}.

The high energies were studied in \cite{Sc2}:
\begin{theorem}\cite{Sc2} \label{high eng thm}
  Assume that $|V|\lesssim \langle x\rangle^{-2-}$, then for any $\lambda_1>0$
  \begin{align*}
	\left\|e^{itH} P_{ac}-e^{itH}\chi(H)P_{ac}  \right\|_{1\rightarrow\infty}\leq C_{\lambda_1} |t|^{-1}.
  \end{align*}
\end{theorem}

\noindent
Therefore, in the proof of Theorem~\ref{big thm} and Theorem~\ref{eval only thm}, it suffices to obtain the stated bounds for the
operator $K_{\lambda_1}$ for some $\lambda_1>0$. Our analysis relies on expansions of the resolvent operator at zero energy following those of \cite{JN}, also see the previous work in \cite{BDG,BGW}. We repeat part of the argument to obtain more flexible and favorable error bounds for our purposes.

We also note that standard spectral theoretic results for $H$ apply.  Under our  assumptions  we have that the spectrum of $H$ can be expressed as the absolutely continuous spectrum, the interval
$[0,\infty)$, and finitely many eigenvalues of finite multiplicity on $(-\infty,0]$.  See
\cite{RS1} for spectral theory and \cite{Stoi} for Birman-Schwinger type
bounds.

Our paper is organized as follows.  We set out the necessary expansions for the resolvent in Section~\ref{sec:exp}.
%\hmm{rewrite this part after we are done with the paper, there are already some changes}
We then study $K_{\lambda_1}$ to establish
Theorem~\ref{big thm} in the case when there is an s-wave resonance at zero
in Section~\ref{sec:swave}.  In Section~\ref{sec:pwave} we establish Theorem~\ref{big thm} in the case of a p-wave resonance or eigenvalue
at zero energy.  In Section~\ref{sec:spectral} we discuss the spectral structure of $-\Delta+V$ at zero energy.  Finally we prove Theorem~\ref{eval only thm} in Section~\ref{sec:weighted}.

\section{Resolvent expansions around zero energy in the case of an s-wave resonance}\label{sec:exp}

In this section, following \cite{JN}, we obtain resolvent expansions around the threshold
$\lambda=0$ in the case when there is only   s-wave  resonance at zero (resonance of the first kind, see Definition~\ref{resondef} below and the remarks following it).
We now introduce some definitions and notation.

\begin{defin}
	We say an operator $T:L^2(\R^2)\to L^2(\R^2)$ with kernel
	$T(\cdot,\cdot)$ is absolutely bounded if the operator with kernel
	$|T(\cdot,\cdot)|$ is bounded from $L^2(\R^2)$ to $L^2(\R^2)$.
\end{defin}

It is worth noting that a Hilbert-Schmidt operator is an absolutely bounded operator.

We say that an absolutely bounded operator $T(\lambda)(\cdot,\cdot)$ is $\mathcal O_1(\lambda^s)$ if the integral
kernel satisfies the following estimates:
\begin{align}\label{O def 1}
	 \big\|\sup_{0<\lambda<\lambda_1} \lambda^{-s}|T(\lambda)(\cdot,\cdot)|\big\|_{L^2\to L^2}\lesssim 1, \,\,\,\,\,\,\,\,\,\,\,\,\,\,
	 \big\|\sup_{0<\lambda<\lambda_1} \lambda^{1-s}|\partial_\lambda T(\lambda)(\cdot,\cdot)|\big\|_{L^2\to L^2}\lesssim 1.
\end{align}
If only the first bound in \eqref{O def 1} holds, we say that $T(\lambda)(\cdot,\cdot)$ is $O(\lambda^s)$.  We also note that we can
replace $\lambda^{-s}$ with $f(\lambda)^{-1}$ in \eqref{O def 1}, in which case we say
$T(\lambda)(\cdot,\cdot)$ is $O(f(\lambda))$.

Recall that
\begin{align*}
	R_0^\pm(\lambda^2)(x,y)=\pm\frac{i}{4} H_0^\pm(\lambda|x-y|)
\end{align*}
where $H_0^\pm$ are the Hankel functions of order zero:
\begin{align}\label{H0}
H_0^{\pm}(z)=J_0(z)\pm iY_0(z).
\end{align}
%See \cite{AS} for example.
From the series expansions for the Bessel functions, see \cite{AS}, as $z\to 0$ we have
\begin{align}
	J_0(z)&=1-\frac{1}{4}z^2+\frac{1}{64}z^4+O(z^6)=1+O(z^2),\label{J0 def}\\
	Y_0(z)&=\frac{2}{\pi}(\log(z/2)+\gamma)J_0(z)+\frac{2}{\pi}\bigg(\frac{1}{4}z^2-\frac{3}{128}z^4+O(z^6)\bigg)\label{Y0 def}\\
&=\frac2\pi \log(z)+O(1).\label{Y0 def2}
\end{align}
We also have the following estimates for the derivatives as $z\to 0$
\begin{align}\label{Jprime}
	 J_0^\prime(z)=O(z),\,\,\,\,\,\,\,J_0^{\prime\prime}(z)=O(1), \,\,\,\,\,\,\, Y_0^\prime(z)=\frac{2}{\pi z}+O(1). 
\end{align}
Further, for $|z|>1 $, we have the representation (see, {\em e.g.}, \cite{AS})
\begin{align}\label{JYasymp2}
	&H_0^\pm(z)= e^{\pm iz} \omega_\pm(z),\,\,\,\, |\omega_{\pm}^{(\ell)}(z)|\lesssim (1+|z|)^{-\frac{1}{2}-\ell},\,\,\,\ell=0,1,2,\ldots.
%\sim \sqrt{\frac{2}{\pi z}}e^{i(z-\frac{1}{4}\pi)}
%	&H_0^\prime(z)\sim \sqrt{\frac{2}{\pi z}}e^{i(z-\frac{1}{4}\pi)}(1+o(1))
\end{align}
This implies that for $|z|>1$
\begin{align}\label{largeJYH}
	\mathcal C(z)=e^{iz} \omega_+(z)+e^{-iz}\omega_-(z), \qquad
	|\omega_{\pm}^{(\ell)}(z)|\lesssim (1+|z|)^{-\frac{1}{2}-\ell},\,\,\,\ell=0,1,2,\ldots,
\end{align}
 for any $\mathcal C\in\{J_0, Y_0\}$ respectively with different $\omega_{\pm}$.

Let $U(x)=1$ if $V(x)\geq 0$ and $U(x)=-1$ if $V(x)<0$, and let $v=|V|^{1/2}$. We have $V=Uv^2$.
We use the symmetric resolvent identity, valid for $\Im\lambda>0$:
\be\label{res_exp}
R_V^\pm(\lambda^2)=  R_0^\pm(\lambda^2)-R_0^\pm(\lambda^2)vM^\pm(\lambda)^{-1}vR_0^\pm(\lambda^2),
\ee
where $M^\pm(\lambda)=U+vR_0^\pm(\lambda^2)v$.
The key issue in the resolvent expansions is the invertibility of the operator $M^\pm(\lambda)$ for small $\lambda$
under various spectral assumptions at zero. Below, we obtain expansions of the operator $M^\pm(\lambda)$ around
$\lambda =0$ using the properties of the free resolvent listed
above. A similar lemma was proved in \cite{Sc2}, however we need to expand the operator further and obtain slightly
more general error bounds. The following operators  arise naturally in the expansion of $M^{\pm}(\lambda)$
(see \eqref{J0 def}, \eqref{Y0 def})
\begin{align}
	G_0f(x)&=-\frac{1}{2\pi}\int_{\R^2} \log|x-y|f(y)\,dy, \label{G0 def}\\
	G_1f(x)&=\int_{\R^2} |x-y|^2f(y)\, dy,\label{G1 def}\\
	G_2f(x)&= \frac1{8\pi}\int_{\R^2} |x-y|^2\log|x-y|f(x)\, dy.\label{G2 def}
\end{align}

\begin{lemma} \label{lem:M_exp}
	For $\lambda>0$ define $M^\pm(\lambda):=U+vR_0^\pm(\lambda^2)v$.
	Let $P=v\langle \cdot, v\rangle \|V\|_1^{-1}$ denote the orthogonal projection onto $v$.  Then
	\begin{align*}
		M^{\pm}(\lambda)=g^{\pm}(\lambda)P+T+M_0^{\pm}(\lambda).
	\end{align*}
	Here $g^{\pm}(\lambda)=a\ln \lambda+z$ where $a\in\R\backslash\{0\}$ and $z\in\mathbb C\backslash\R$,
	and $T=U+vG_0v$ where $G_0$ is an
	integral operator defined in \eqref{G0 def}.
	Further, for any $\frac{1}{2}\leq k<2$,
	\begin{align*}
		M_0^{\pm}(\lambda)=\mathcal O_1(\lambda^k)
	\end{align*}
	if $v(x)\lesssim \langle x\rangle^{-\beta}$ for some $\beta>1+k$. Moreover,
	\begin{align}\label{M0_defn}
		M_0^{\pm}(\lambda)= g_1^{\pm}(\lambda)vG_1v
		+\lambda^2 vG_2 v+M_1^{\pm}(\lambda).
	\end{align}
	Here $G_1$, $G_2$ are integral operators defined in
	\eqref{G1 def}, \eqref{G2 def},  and
	$g_1^{\pm}(\lambda)=\lambda^2 (\alpha \log \lambda+\beta_{\pm})$ where $\alpha \in\R\backslash\{0\}$ and $\beta_\pm \in\mathbb C\backslash\R$.
  Further, for any $2<\ell<4$,
	\begin{align*}
		M_1^{\pm}(\lambda)=\mathcal O_1(\lambda^\ell)
	\end{align*}
	if $\beta>1+\ell$.
\end{lemma}

\begin{proof}
	The first part with $k=\frac{1}{2}$ was proven in  \cite[Lemma 5]{Sc2}.
	To obtain the expansions recall that, for $\lambda>0$,
	$$
		R_0^{\pm}(\lambda^2)(x,y)=\pm \frac{i}{4} H_0^{\pm}(\lambda|x-y|).
	$$
	Using the definition of $H_0^{\pm}(z)$, and the expansions \eqref{J0 def} and
	\eqref{Y0 def} around $z=0$,  we have
	\begin{align}
		\pm \frac{i}{4} H_0^{\pm}(z)&=\pm \frac{i}{4}  J_0(z)-\frac{1}{4} Y_0(z) =-\frac{1}{2\pi}\log(z/2) \pm\frac{i}4 -\frac\gamma{2\pi}   +  \alpha z^2 \log z+\beta_\pm z^2+O(z^4\log z)
		\label{Hankel exp}\\
&= -\frac{1}{2\pi}\log(z/2) \pm\frac{i}4 -\frac\gamma{2\pi}   + O( z^2 \log z) \label{exp1}
	\end{align}
	with $\alpha=1/8\pi$ and $\beta_\pm\in \mathbb C$.  The expansions are now obtained by setting $z=\lambda|x-y|$.
	In particular, we see
	\begin{align}\label{g form}
		g^{\pm}(\lambda)=-\|V\|_1\bigg(\frac{1}{2\pi}\log(\lambda/2)+\frac{1}{2\pi}\gamma\mp \frac{i}{4}\bigg).
	\end{align}
	Noting that
	\begin{align*}
		M_0^{\pm}(\lambda)&=[U+vR_0^{\pm}(\lambda^2)v]-[g^{\pm}(\lambda)P+U+vG_0v],\\
		M_1^{\pm}(\lambda)&=[U+vR_0^{\pm}(\lambda^2)v]-[g^{\pm}(\lambda)P+U+vG_0v+g_1^{\pm}(\lambda)
		vG_1 v+ \lambda^2 vG_2v].
	\end{align*}
	Using \eqref{exp1} and \eqref{Hankel exp}  for $M_0$ and $M_1$ respectively, we obtain for $z=\lambda|x-y|<1$ that
	\begin{align*}
		|M_0^{\pm}(\lambda)(x,y) \chi_{\{\lambda|x-y|<1\}}|&\lesssim v(x)v(y) (\lambda|x-y|)^2\log(\lambda|x-y|)  
\lesssim v(x)v(y) (\lambda|x-y|)^{2-},\\
		|M_1^{\pm}(\lambda)(x,y)\chi_{\{\lambda|x-y|<1\}}|&\lesssim v(x)v(y) (\lambda|x-y|)^4\log(\lambda|x-y|) \lesssim v(x)v(y)(\lambda|x-y|)^{4-}.
	\end{align*}
	For large $z$, using the expansion of the Hankel function about $z=\infty$,  recall \eqref{JYasymp2},
	we have
	$|H_0^{\pm}(z)|\lesssim 1$ and $|\frac{d}{dz}H_0^{\pm}(z)|\lesssim z^{-1/2}$.
	 So that for large $z>1$, for $M_0^{\pm}(z)$ the $\log z$ term dominates and for
	$M_1^{\pm}(z)$ the $z^2\log z$ term in \eqref{Hankel exp} dominates, and we have
	\begin{align*}
		|M_0^{\pm}(\lambda)(x,y)\chi_{\{\lambda|x-y|>1\}}|&\lesssim v(x)v(y)  \log(\lambda|x-y|)\chi_{\{\lambda|x-y|>1\}} \lesssim v(x)v(y)
		(\lambda|x-y|)^{0+}\chi_{\{\lambda|x-y|>1\}},\\
		|M_1^{\pm}(\lambda)(x,y)\chi_{\{\lambda|x-y|>1\}}|&   % \lesssim v(x)v(y)  (\lambda|x-y|)^2 \log(\lambda|x-y|)\chi_{\{\lambda|x-y|>1\}} 
\lesssim v(x)v(y)(\lambda|x-y|)^{2+}\chi_{\{\lambda|x-y|>1\}}.
	\end{align*}
	Hence, for any $0<k<2$, and for any $2<\ell<4$ we have
	\begin{align*}
		|M_0^{\pm}(\lambda)(x,y)| &\lesssim v(x)v(y)\big[(\lambda|x-y|)^{2-}\chi_{\{\lambda|x-y|<1\}}
		+ (\lambda|x-y|)^{0+}\chi_{\{\lambda|x-y|>1\}} \big]   \lesssim v(x)v(y)(\lambda|x-y|)^{k},\\
		|M_1^{\pm}(\lambda)(x,y)| &\lesssim v(x)v(y)\big[(\lambda|x-y|)^{4-}\chi_{\{\lambda|x-y|<1\}}
		+ (\lambda|x-y|)^{2+}\chi_{\{\lambda|x-y|>1\}}\big]  \lesssim v(x)v(y) (\lambda|x-y|)^{\ell}.
	\end{align*}
	 This yields the claim for $M_0$ and $M_1$ since $v(x)v(y)|x-y|^{\ell}$ is Hilbert-Schmidt from $L^{2}(\R^2)$ to
	$L^{2}(\R^2)$ for $\beta>1+\ell$.  For $\lambda$-derivatives,  we note that
	\begin{align}\label{partial R0}
		|\partial_\lambda R_0^{\pm}(\lambda^2)(x,y)|\lesssim \bigg(\frac{|x-y|}{\lambda}\bigg)^{\frac{1}{2}},
	\end{align}
	and
	\begin{align*}
		\partial_\lambda F(\lambda|x-y|)=|x-y|\partial_zF(z)\bigg|_{z=\lambda|x-y|}.
	\end{align*}
	For the terms in $M_0$ and $M_1$ other than $R_0$, the effect of $\partial_\lambda$ is comparable
	to division by $\lambda$.
	However, due to \eqref{partial R0}, on $\lambda|x-y|>1$ we have for any $k\geq \frac{1}{2}$,
	\begin{align*}
		|\partial_\lambda M_0(\lambda)(x,y)|\lesssim v(x)v(y)\bigg[ \bigg(\frac{|x-y|}{\lambda}\bigg)^{\frac{1}{2}}
		+\lambda^{-1}\bigg] \lesssim v(x)v(y) \lambda^{k-1} |x-y|^{k}.
	\end{align*}
	Similarly,
	\begin{align*}
		|\partial_\lambda M_1(\lambda)(x,y)|&\lesssim v(x)v(y)\bigg[ \bigg(\frac{|x-y|}{\lambda}\bigg)^{\frac{1}{2}}
		+\lambda^{-1}(\lambda|x-y|)^\ell \bigg] \lesssim v(x)v(y) \lambda^{\ell-1}|x-y|^\ell. \qedhere
	\end{align*}
\end{proof}

We now give the definition of resonances from \cite{JN}, also see \cite{Sc2}. Recall that
$Q:=\mathbbm{1}-P$.

\begin{defin}\label{resondef}\begin{enumerate}
\item We say zero is a regular point of the spectrum
of $H = -\Delta+ V$ provided $ QTQ=Q(U + vG_0v)Q$ is invertible on $QL^2(\mathbb R^2)$.

\item Assume that zero is not a regular point of the spectrum. Let $S_1$ be the Riesz projection
onto the kernel of $QTQ$ as an operator on $QL^2(\mathbb R^2)$.
Then $QTQ+S_1$ is invertible on $QL^2(\mathbb R^2)$.  Accordingly, we define $D_0=(QTQ+S_1)^{-1}$ as an operator
on $QL^2(\R^2)$.
We say there is a resonance of the first kind at zero if the operator $T_1:= S_1TPTS_1$ is invertible on
$S_1L^2(\mathbb R^2)$.

\item We say there is a resonance of the second kind at zero if $T_1$ is not invertible on
$S_1L^2(\R^2)$ but
$T_2:=S_2vG_1vS_2$ is invertible
on $S_2L^2(\R^2)$, where $S_2$ is the Riesz projection onto the kernel of $T_1$ (recall the definition of
$G_1$ and $G_2$ in \eqref{G1 def} and \eqref{G2 def}).

\item Finally, if $T_2$ is not invertible on $S_2L^2(\R^2)$, we say there is a resonance of the third kind at zero.
We note that in this case the operator $T_3:=S_3vG_2vS_3$ is always invertible on $S_3L^2$, where $S_3$ is the Riesz
projection onto the kernel of $T_2$ (see (6.41) in \cite{JN} or Section~\ref{sec:spectral} below).

\end{enumerate}

\end{defin}

\noindent
{\bf Remarks.} i) In \cite{JN}, it is noted that the projections
$S_1-S_2$, $S_2-S_3$ and $S_3$  correspond  to   s-wave  resonances,  p-wave  resonances,
and zero eigenspace respectively. In particular, resonance of the first kind means that there is only an s-wave resonance at zero.
Resonance of the second kind means that there is a p-wave resonance, and there may or may not be an s-wave resonance.
Finally, resonance of the third kind means that zero is an eigenvalue, and there may or may not be s-wave and p-wave resonances.
We characterize these projections in Section~\ref{sec:spectral}. \\
ii) Since $QTQ$ is self-adjoint, $S_1$ is the orthogonal projection onto the kernel of $QTQ$, and we have
(with $D_0=(QTQ+S_1)^{-1}$) $$S_1D_0=D_0S_1=S_1.$$
This statement also valid for $S_2$ and $(T_1+S_2)^{-1}$, and for $S_3$ and $(T_2+S_3)^{-1}$.\\
iii) The operator $QD_0Q$ is absolutely bounded in $L^2$. This was proved in
Lemma~8 of \cite{Sc2} in the case $S_1=0$.
With minor modifications, the same proof works in our case, too.\\
iv) The operators with kernel $vG_i v$ are Hilbert-Schmidt operators on $L^2(\R^2)$ if $v(x)\les \langle x\rangle^{-\beta}$
for $\beta>\frac{3}{2}$ if $i=1$ and $\beta>3$ for $i=2,3$.

To invert $M^\pm(\lambda)=U+vR_0^\pm(\lambda^2)v$, for small $\lambda$, we will use the
following  lemma (see Lemma 2.1  in \cite{JN}) repeatedly.
\begin{lemma}\label{JNlemma}
Let $A$ be a closed operator on a Hilbert space $\mathcal{H}$ and $S$ a projection. Suppose $A+S$ has a bounded
inverse. Then $A$ has a bounded inverse if and only if
$$
B:=S-S(A+S)^{-1}S
$$
has a bounded inverse in $S\mathcal{H}$, and in this case
$$
A^{-1}=(A+S)^{-1}+(A+S)^{-1}SB^{-1}S(A+S)^{-1}.
$$
\end{lemma}

We will apply this lemma with $A=M^\pm(\lambda)$ and $S=S_1$. Thus, we need to show that $M^{\pm}(\lambda)+S_1$
has a bounded inverse in $L^2(\mathbb R^2)$ and
\begin{align}\label{B defn}
  B_{\pm}=S_1-S_1(M^\pm(\lambda)+S_1)^{-1}S_1
\end{align}
has a bounded inverse in $S_1L^2(\mathbb R^2)$.
We prove these claims and obtain expansions for the inverses for each type of resonance
in Lemma~\ref{M+S1inverse}, Proposition~\ref{prop:Binverse1}, and Proposition~\ref{prop:Binverse2} below.

\begin{lemma}\label{M+S1inverse}

	Suppose that zero is not a regular point of the spectrum of  $H=-\Delta+V$, and let $S_1$ be the corresponding
	Riesz projection. Then for   sufficiently small $\lambda_1>0$, the operators
	$M^{\pm}(\lambda)+S_1$ are invertible for all $0<\lambda<\lambda_1$ as bounded operators on $L^2(\R^2)$.
	Further, one has
	\begin{align}
	\label{M plus S size}
        	=h_{\pm}(\lambda)^{-1}S+QD_0Q+\mathcal O_1(\lambda^k),
	\end{align}
	for any $\frac{1}{2}\leq k<2$ if $v(x)\lesssim \langle x\rangle^{-(1+k)-}$.
	Here
   $h_+(\lambda)=\overline{h_-(\lambda)}=a\ln \lambda+z$ (with $a\in\R\backslash \{0\}$ and
  $z\in\mathbb C$,  $\Im z\neq 0$), and
  \be\label{S_defn}
    S=\left[\begin{array}{cc} P & -PTQD_0Q\\ -QD_0QTP & QD_0QTPTQD_0Q
		\end{array}\right]
  \ee
is a finite-rank operator with real-valued kernel.
\end{lemma}

\begin{proof}

	  We will give the proof for $M^+$ and drop the superscript ``$+$" from formulas.  Using Lemma~\ref{lem:M_exp}, we write $M(\lambda)+S_1$ with respect to the decomposition
	$L^2(\R^2)=PL^2(\R^2)\oplus QL^2(\R^2)$.
	\begin{align*}
		M(\lambda)+S_1=\left[ \begin{array}{cc} g(\lambda)P+P(T+S_1)P & P(T+S_1)Q\\
		Q(T+S_1)P & Q(T+S_1)Q
		\end{array}\right]+M_0(\lambda).
	\end{align*}
	Noting that $Q\geq S_1$, we have $S_1P=PS_1=0$.  Therefore,
	\begin{align*}
		M(\lambda)+S_1=\left[ \begin{array}{cc} g(\lambda)P+PTP & PTQ\\
		QTP & Q(T+S_1)Q
		\end{array}\right]+M_0(\lambda).
	\end{align*}
	Denote the matrix component of the above equation by $A(\lambda)=\{a_{ij}(\lambda)\}_{i,j=1}^{2}$.

    Since  $Q(T+S_1)Q$ is invertible, by the Fehsbach formula invertibility of  $A(\lambda)$ hinges upon the existence
    of $d=(a_{11}-a_{12}a_{22}^{-1}a_{21})^{-1}$. Denoting $D_0=(Q(T+S_1)Q)^{-1}:QL^2\to QL^2$, we have
	\begin{align*}
		d= (g(\lambda)P+PTP-PTQD_0QTP)^{-1} =h(\lambda)^{-1}P
	\end{align*}
	with $h(\lambda)=g(\lambda)+Tr(PTP-PTQD_0QTP)=a\ln(\lambda)+z$, with $a\in\R$ and $z\in\mathbb C$.
	This follows from \eqref{g form} and the fact that $Tr(PTP-PTQD_0QTP)$ is $\lambda$ independent and
	real-valued, as the kernels of $T$, $QD_0Q$ and $v$ are real-valued.
	Therefore,  $d$ exists if $\lambda$ is sufficiently small.

Thus, by the
	Fehsbach formula,
	\begin{align}\nonumber
		A(\lambda)^{-1}&=\left[\begin{array}{cc} d & -da_{12}a_{22}^{-1}\\
		-a_{22}^{-1}a_{21}d & a_{22}^{-1}a_{21}da_{12}a_{22}^{-1}+a_{22}^{-1}
		\end{array}\right]\\
		&=h^{-1}(\lambda)\left[\begin{array}{cc} P & -PTQD_0Q\\ -QD_0QTP & QD_0QTPTQD_0Q
		\end{array}\right]+QD_0Q \nonumber\\
		&=:h^{-1}(\lambda)S+QD_0Q. \label{Ainverse}
	\end{align}
    Note that $S$ has rank at most two. This and the absolute boundedness of $QD_0Q$ imply that $A^{-1}(\lambda)= \mathcal O_1(1)$.

    Finally, we write
    $$
    M(\lambda)+S_1=A(\lambda)+M_0(\lambda)=[\mathbbm{1}+M_0(\lambda) A^{-1}(\lambda)] A(\lambda).
    $$
Since $A^{-1}(\lambda)= \mathcal O_1(1)$ and, by Lemma~\ref{lem:M_exp},
    $M_0(\lambda)=\mathcal O_1(\lambda^{k})$ provided $|v(x)|\lesssim \langle x\rangle^{-(1+k)-}$, we obtain
\be\label{M plus S}
        (M^{\pm}(\lambda)+S_1)^{-1}=A_{\pm}^{-1}(\lambda) \big[\mathbbm{1}+M_0^\pm(\lambda) A_\pm^{-1}(\lambda)\big]^{-1}=h_{\pm}(\lambda)^{-1}S+QD_0Q+\mathcal O_1(\lambda^k),
\ee
  by a Neumann series expansion.
\end{proof}

We now prove the invertibility of the operators $B_{\pm}=S_1-S_1(M^\pm(\lambda)+S_1)^{-1}S_1$.

\begin{prop}\label{prop:Binverse1} Assume that $|v(x)|\lesssim \langle x\rangle^{-1-k-}$ for some
$k\in [\frac{1}{2},2)$. Then, in the case of a resonance of the first kind, $B_{\pm}$ is invertible on
$S_1L^2(\R^2)$ and we have
	\begin{align}\label{firstk}
		B^{-1}_{\pm} =-h_{\pm}(\lambda) D_1 +\mathcal O_1(\lambda^k),
	\end{align}
where $D_1=T_1^{-1}=(S_1TPTS_1)^{-1}$, and $h_{\pm}(\lambda)$ is as in Lemma~\ref{M+S1inverse}.
\end{prop}

\begin{proof}
	We again prove the case of the ``$+$" superscripts and subscripts and omit them from the notation.
    Using Lemma~\ref{M+S1inverse}, we obtain
	\begin{align*}
		B=S_1-S_1(h(\lambda)^{-1}S+QD_0Q)S_1+ \mathcal O_1(\lambda^{k+}).
	\end{align*}
	Recall that $S_1 D_0=D_0 S_1=S_1$.
	Further, from the definition \eqref{S_defn} of $S$,  and the fact that $S_1P=PS_1=0$, we obtain
$S_1SS_1=S_1 TPT S_1=T_1$. Therefore
	\begin{align}\label{B defn 2}
		B=-h(\lambda)^{-1}S_1SS_1+\mathcal O_1(\lambda^{k+}) =-h(\lambda)^{-1} T_1+ \mathcal O_1(\lambda^{k+}).
	\end{align}	
Recall that by the definition of a resonance of the first kind, the leading term $T_1$ in the
definition of $B$ is invertible on $S_1 L^2(\R^2)$.  Therefore, for sufficiently small $\lambda$,
	\begin{align*}
		B^{-1} =-h(\lambda)[T_1-h(\lambda) \mathcal O_1(\lambda^{k+}) ]^{-1}=-h(\lambda)[T_1+\mathcal O_1(\lambda^{k+})]^{-1} =-h(\lambda) D_1 +  \mathcal O_1(\lambda^{k}). 
	\end{align*}
\end{proof}

Combining Lemma~\ref{JNlemma}, Lemma~\ref{M+S1inverse}, and Proposition~\ref{prop:Binverse1}, we obtain

\begin{corollary}\label{M exp cor1} Assume that $|v(x)|\lesssim \langle x\rangle^{-1-k-}$ for some
$k\in [\frac{1}{2},2)$. Then in the case of a resonance of the first kind, we have
	\begin{multline*}
		M^{\pm}(\lambda)^{-1} =-h_{\pm}(\lambda)S_1D_1S_1-SS_1D_1S_1-S_1D_1S_1S\\
		 -h_{\pm}(\lambda)^{-1}SS_1D_1S_1S+
		h_{\pm}(\lambda)^{-1} S+QD_0Q +\mathcal O_1(\lambda^k),
	\end{multline*}
provided that $\lambda$ is sufficiently small.
\end{corollary}

\begin{proof}
    Combining Lemma~\ref{JNlemma}, Lemma~\ref{M+S1inverse}, and Proposition~\ref{prop:Binverse1}, we have
	%\begin{align*}
\begin{multline*}
		M^{\pm}(\lambda)^{-1} =(M^{\pm}(\lambda)+S_1)^{-1}+(M^{\pm}(\lambda)+S_1)^{-1}S_1B^{-1}S_1
		(M^{\pm}(\lambda)+S_1)^{-1}\\
		 =h_{\pm}(\lambda)^{-1}S+QD_0Q  -h(\lambda)\big(h_{\pm}(\lambda)^{-1}S+QD_0Q\big) S_1D_1
		S_1 \big(h_{\pm}(\lambda)^{-1}S+QD_0Q\big) +\mathcal O_1(\lambda^k)\\
=-h_{\pm}(\lambda)S_1D_1S_1-SS_1D_1S_1-S_1D_1S_1S-h_{\pm}(\lambda)^{-1}SS_1D_1S_1S\\ +
h_{\pm}(\lambda)^{-1} S+QD_0Q +\mathcal O_1(\lambda^k).
	\end{multline*} %\end{align*}
Here we used the fact that $S_1QD_0Q=QD_0QS_1=S_1$.
\end{proof}

\noindent
{\bf Remark.}  Under the conditions of Corollary~\ref{M exp cor1}, the resolvent identity
	\begin{align}\label{resolvent id}
	    R_V^{\pm}(\lambda^2)=R_0^{\pm}(\lambda^2)-R_0^{\pm}(\lambda^2) v M^{\pm}(\lambda)^{-1}v
	    R_0^{\pm}(\lambda^2)
	\end{align}
	holds as an  operator identity  on $L^{2,\frac{1}{2}+}(\R^2)\to L^{2,-\frac{1}{2}-}(\R^2)$, as in the limiting absorption principle, \cite{agmon}.

\section{Resonance of the first kind}\label{sec:swave}
In this section, we establish the estimates needed to prove Theorem~\ref{big thm}.  We assume that there
is a resonance of the first kind, $\lambda_1$ is sufficiently small (so that the analysis in the previous section is valid), and that
$v(x)\lesssim \langle x\rangle^{-(1+k)-}$ for $k=1$, or equivalently $|V(x)|\lesssim \langle x\rangle^{-4-}$.
 It suffices to  prove that
\begin{theorem}\label{swave thm} Under the conditions above, we have
\be\label{disp}
|\langle K_{\lambda_1} f, g\rangle| \lesssim |t|^{-1},
\ee
for Schwartz functions $f$ and $g$ with $\|f\|_1=\|g\|_1=1$.
\end{theorem}

 This theorem will be established in Propositions~\ref{prop:summand2}, \ref{S and S_1 terms lemma}, \ref{S prop}, and \ref{Error term lemma}. All statements in this section are valid under the conditions above.
\begin{prop}\label{prop:summand2}
	The contribution of the first term in Corollary~\ref{M exp cor1} in \eqref{klambda} satisfies \eqref{disp}. More explicitly, we have the bound
\begin{align*}
  	\bigg|\int_{\R^8} \int_0^\infty e^{it\lambda^2} \lambda \chi(\lambda)\mathcal{K}(\lambda,p,q)
	v(x_1)S_1D_1S_1(x_1,y_1)
   	v(y_1)  f(x)g(y) d\lambda dx_1 dy_1 dx dy\bigg|\lesssim |t|^{-1},
\end{align*}
where $p=|x-x_1|$, $q=|y-y_1|$, and
\begin{multline}  
	\mathcal{K}(\lambda,p,q) =h^+(\lambda) H^{+}_0(\lambda p)H_0^+(\lambda q)-h^-(\lambda)
	H^{-}_0(\lambda p)H_0^-(\lambda q) \\
	 =2ia\log(\lambda) [Y_0(\lambda p)J_0(\lambda q)+J_0(\lambda p)Y_0(\lambda q)]
+2z [J_0(\lambda p)J_0(\lambda q)+Y_0(\lambda p)Y_0(\lambda q)].\label{K defn}
\end{multline}
\end{prop}
To prove this proposition, we need  to consider the high and low energy contributions of the Bessel functions separately. To this end we use the partitions of unity
 $1=\chi(\lambda |y-y_1|)+\tilde{\chi}(\lambda|y-y_1|)$ and
$1=\chi(\lambda |x-x_1|)+\tilde{\chi}(\lambda|x-x_1|)$.   We divide the proof of Proposition~\ref{prop:summand2} into
Lemmas~\ref{low log YJ lemma}, \ref{low high log YJ lemma},
\ref{high high log YJ lemma} and
their respective corollaries, Corollaries~\ref{low YJ lemma2}, \ref{low high YJ lemma2},
 due to the various terms arising in \eqref{K defn}.

For the low energy parts, the following lemma will be useful:

\begin{lemma}\label{lem:FG}
	
	Let $p=|x-x_1|$, $q=|x|+1$, and
	\begin{align*}
		F(\lambda,x,x_1)&:=\chi(\lambda p)Y_0(\lambda p)-\chi(\lambda q) Y_0(\lambda q ),\\
		G(\lambda,x,x_1)&:=\chi(\lambda p)J_0(\lambda p)-\chi(\lambda q) J_0(\lambda q).
	\end{align*}
	Then for any $\tau\in [0,1]$ and $\lambda \leq 2\lambda_1$ we have
	\begin{align*}
		|G(\lambda,x,x_1)|&\lesssim  \lambda^{\tau} \langle x_1\rangle^{\tau},
		\qquad \,\,\,\,\,\,|\partial_\lambda G(\lambda,x,x_1)|\lesssim \langle x_1\rangle^\tau\lambda^{\tau-1},\\
		|F(\lambda,x,x_1)|&\leq \int_0^{2\lambda_1} |\partial_\lambda F(\lambda,x,x_1)| d\lambda
		+|F(0+,x,x_1)|\lesssim k(x,x_1),\,\,\,\,\,\,\,\,\,\,\,
|\partial_\lambda F(\lambda,x,x_1)| \lesssim \frac1\lambda. 
	\end{align*}
Here $k(x,x_1):=1+\log^+|x_1|+\log^- |x-x_1|$, where $\log^- y:=\chi_{\{0<y<1\}} |\log y|$ and $\log^+ y:=\chi_{\{y>1\}} \log y$.G
\end{lemma}

\begin{proof}
We start with $G$. Let $g(s):=\chi(s)J_0(s)$. We have $g^\prime(s) =O(1).$ Therefore,  by the  mean value theorem and the boundedness of $g$, we have
\begin{align*}
	|G(\lambda,x,x_1)|\lesssim \min(\lambda |p-q|,1) \lesssim \min( \lambda \langle x_1\rangle,1)
	\lesssim \lambda^{\tau} \langle x_1\rangle^{\tau},
\end{align*}
for any $0\leq \tau\leq 1.$

Now consider
\begin{align*}
	|\partial_\lambda G(\lambda,x,x_1)|=|p g^\prime(\lambda p)-q g^\prime(\lambda q)|.
\end{align*}
Let $g_1(s)=sg^\prime(s)$. We have
$|g_1(s)|\lesssim 1$ and $|g_1^\prime(s)|\lesssim 1$. Therefore, by the mean value theorem and
the boundedness of $g_1$, we have
\begin{align*}
	|\partial_\lambda G(\lambda,x,x_1)|=\bigg|\frac{g_1(\lambda p)-g_1(\lambda q)}{\lambda}\bigg|
	\lesssim \bigg|\frac{\min(\lambda|p-q|,1)}{\lambda}\bigg|\lesssim
	\langle x_1\rangle^\tau\lambda^{\tau-1},
\end{align*}
for any $\tau\in [0,1]$.

The bounds for $F$ were obtained in \cite{Sc2}. We repeat them for completeness.
Note that $F(0+,x,x_1)=\log\Big(\frac{|x-x_1|}{|x|+1}\Big)+c\lesssim k(x,x_1)$. Therefore it suffices to bound
\begin{align}\label{F prime int}
    \int_0^{2\lambda_1} |\partial_\lambda F(\lambda,x,x_1)| d\lambda
    & \lesssim \int_0^{2\lambda_1} p|\chi^\prime(\lambda p)\log(\lambda p)| d\lambda + \int_0^{2\lambda_1} q|\chi^\prime(\lambda q)\log(\lambda q)| d\lambda \\
    &+\int_0^{2\lambda_1} \frac{1}{\lambda} |\chi(\lambda p)-\chi(\lambda q)| d\lambda. 
\end{align}
By inspecting the integrands on the right hand side, we see that $|\partial_\lambda F|$ is bounded by $1/\lambda$. To obtain the statement for $|F|$ first note that,
since $\chi^\prime$ is supported in the set $[\lambda_1/2,2\lambda_1]$, the first  line  in \eqref{F prime int} is $\lesssim 1$.
To estimate the second line note that $\chi(\lambda p)-\chi(\lambda q)$ is supported on the set
$[\frac{\lambda_1}{2 p}, 2\frac{\lambda_1}{q}]$, which implies that the last line is
$\lesssim \big|\log\Big(\frac{|x-x_1|}{|x|+1}\Big)\big|\lesssim k(x,x_1)$.
\end{proof}

\begin{lemma}\label{low log YJ lemma}

  We have the bound
  \begin{multline}\label{low low bound}
    \bigg| \int_{\R^8} \int_0^\infty e^{it\lambda^2} \lambda \chi(\lambda)\log \lambda \chi(\lambda|x-x_1|)Y_0(\lambda|x-x_1|)v(x_1)
     S_1D_1S_1(x_1,y_1) v(y_1) \\ J_0(\lambda|y-y_1|)
    \chi(\lambda|y-y_1|) \, d\lambda f(x) g(y) \, dx_1\, dy_1\, dx\, dy\bigg|\lesssim |t|^{-1}.
  \end{multline}
\end{lemma}

\begin{proof}
	Since $S_1\leq Q$ are projections and $Q$ is the projection orthogonal to $v$, we have
  	\begin{align}
   	 	\int_{\R^4} v(x) [S_1 D_1 S_1](x,y) h(y) dx\, dy
		=\int_{\R^4} h(x) [S_1 D_1 S_1](x,y) v(y)\, dx\, dy=0\label{Q trick}
  	\end{align}
  	for all $h\in L^2(\R^2)$.  As such, we can subtract functions of $x$ (resp. $y$) only from $\chi Y_0$
	(resp.  $\chi J_0$) in the integrand of \eqref{low low bound}. We use the functions defined in Lemma~\ref{lem:FG}. Thus we replace $\chi(\lambda|x-x_1|)Y_0(\lambda|x-x_1|)$ with $F(\lambda,x,x_1)$ and
	$\chi(\lambda|y-y_1|)J_0(\lambda|y-y_1|)$ with $G(\lambda,y,y_1)$ on the left hand side of \eqref{low low bound}. Therefore the $\lambda$ integral of \eqref{low low bound} is equivalent to
  \begin{align}\label{FG log integral}
      \int_0^\infty e^{it\lambda^2} \lambda \chi(\lambda) \log (\lambda) F(\lambda, x,x_1) G(\lambda, y, y_1)\, d\lambda.
  \end{align}
  We integrate by parts once to get
  \begin{align}
    \eqref{FG log integral} &\lesssim  |t|^{-1}\int_0^{\infty} [\log \lambda\chi'(\lambda)+\lambda^{-1}]|F(\lambda,x,x_1)|
    |G(\lambda,y,y_1)|\, d\lambda\label{no der FG}\\
    &+|t|^{-1} \int_0^\infty |\chi(\lambda)\log \lambda||\partial_\lambda F(\lambda, x,x_1)| |G(\lambda,y,y_1)|\,
    d\lambda \label{F der}\\
    &+|t|^{-1} \int_0^\infty |\chi(\lambda)\log \lambda|| F(\lambda, x,x_1)| |\partial_\lambda G(\lambda,y,y_1)|\,
    d\lambda.\label{G der}
  \end{align}
There is no boundary term since, by Lemma~\ref{lem:FG}, we have that $F(0+,y_1,y)  \lesssim k(x,x_1)$ and $G(0,y,y_1)=0$.
  From Lemma~\ref{lem:FG} again, we have for any $\tau\in(0,1]$
  \begin{align*}
    \eqref{no der FG}  \lesssim \int_0^{2\lambda_1} [\log \lambda +\lambda^{-1}] k(x,x_1)
    \lambda^{\tau} \langle y_1\rangle^{\tau}\, d\lambda
    \lesssim \langle y_1\rangle^{\tau} k(x,x_1).
  \end{align*}
  Taking $\tau=0+$, this term now contributes the following to \eqref{low low bound},
  \begin{align}
    &\lesssim|t|^{-1} \int_{\R^8} k(x,x_1)v(x_1) | D_1 (x_1,y_1)| v(y_1) \langle y_1\rangle^{0+} |f(x)|
    |g(y)|\, dx_1\, dy_1\, dx\, dy\label{1st der bound}\\
    &\lesssim |t|^{-1} \sup_{x\in \R^2} \|k(x,\cdot) v(\cdot)\|_2 \|| D_1 |\|_{2\to 2}
    \| v\|_{L^{2,0+}} \|f\|_1 \|g\|_1 \lesssim |t|^{-1}.\nn
  \end{align}
  For the case of \eqref{F der} and \eqref{G der}, we again note the bounds in Lemma~\ref{lem:FG},
	and that on the support of $\chi(\lambda)$,
  	$|\lambda^{\tau} \log \lambda|\lesssim 1$ for any $\tau>0$.  The desired bound follows as in \eqref{1st der bound}.
\end{proof}

We also need the following bounds taking care of the contributions of the remaining terms in \eqref{K defn}:

\begin{corollary}\label{low YJ lemma2}

  For $\mathcal C(z)=J_0(z)$ or $\mathcal C(z)=Y_0(z)$, we have the bound
  \begin{multline}
  \label{low low bound2}  \bigg| \int_{\R^8} \int_0^\infty  e^{it\lambda^2} \lambda \chi(\lambda)\chi(\lambda|x-x_1|)
    \mathcal C(\lambda|x-x_1|)v(x_1)
    S_1D_1S_1(x_1,y_1) v(y_1) \\
    \mathcal C(\lambda|y-y_1|) \chi(\lambda|y-y_1|) \, d\lambda f(x) g(y) \, dx_1\, dy_1\, dx\, dy\bigg|
    \lesssim |t|^{-1}.
  \end{multline}
\end{corollary}

\begin{proof}

  Using the notation of Lemma~\ref{low log YJ lemma}, we need to bound
  \begin{align*}
    \bigg|\int_0^\infty e^{it\lambda^2} \lambda \chi(\lambda) F(\lambda, x,x_1) F(\lambda, y,y_1)\, d\lambda\bigg|,
  \end{align*}
  and the similar term when $F$ is replaced by $G$.  This follows easily from
  one integration by parts and the bounds of Lemma~\ref{lem:FG} as
  in the previous lemma.
\end{proof}

We now need to bound the resulting terms when one of the Bessel functions is supported on large energies.
The following variation of stationary phase from \cite{Sc2} will be useful in the analysis.  For completeness we give the proof.
\begin{lemma}\label{stat phase}

	Let $\phi'(0)=0$ and $1\leq \phi'' \leq C$.  Then,
  	\begin{align*}
    		\bigg| \int_{-\infty}^{\infty} e^{it\phi(\lambda)} a(\lambda)\, d\lambda \bigg|
    		\lesssim \int_{|\lambda|<|t|^{-\frac{1}{2}}} |a(\lambda)|\, d\lambda
    		+|t|^{-1} \int_{|\lambda|>|t|^{-\frac{1}{2}}} \bigg( \frac{|a(\lambda)|}{|\lambda^2|}+
    		\frac{|a'(\lambda)|}{|\lambda|}\bigg)\, d\lambda.
  	\end{align*}

\end{lemma}

\begin{proof}

Let $\eta\in C_c^{\infty}(\R)$ be such that $\eta(x)=1$ if $|x|<1$ and $\eta(x)=0$ if $|x|>2$.  Let
$\eta_2(x)=\eta(x/2|t|^{-1/2})$.  Writing $1=\eta_2+(1-\eta_2)$, we rewrite the integral as follows
\begin{align*}
    \bigg| \int_{-\infty}^{\infty} e^{it\phi(\lambda)} a(\lambda)\, d\lambda \bigg|
    &\lesssim\bigg| \int_{-\infty}^{\infty} e^{it\phi(\lambda)} a(\lambda) \eta_2(\lambda)\, d\lambda \bigg|
    +\bigg| \int_{-\infty}^{\infty} e^{it\phi(\lambda)} a(\lambda)(1-\eta_2(\lambda))\, d\lambda \bigg|
\end{align*}
The first term is bounded as in the claim since supp$(\eta_2)=[-|t|^{-\frac{1}{2}},|t|^{-\frac{1}{2}}]$.
For the second term, we integrate by parts once in $\lambda$ to bound with
\begin{align*}
  |t|^{-1}\bigg| \int_{-\infty}^{\infty} e^{it\phi(\lambda)} &\frac{d}{d\lambda}\bigg( \frac{a(\lambda)(1-\eta_2(\lambda))}
  {\phi'(\lambda)} \bigg)\, d\lambda \bigg|
\end{align*}
By Taylor's Theorem,
\begin{align*}
  \phi'(\lambda)=\phi'(0)+\lambda \phi''(c)=\lambda \phi''(c)\approx \lambda,
\end{align*}
Considering the terms when the derivative acts on $a(\lambda)$, $1-\eta_2(\lambda)$ and $1/\phi'(\lambda)$ finishes
the proof.
\end{proof}

In addition we have the following high-energy analogue of Lemma~\ref{lem:FG}.  In light of the high energy representations
of the Bessel functions \eqref{largeJYH}, recall that for $\mathcal C\in \{J_0,Y_0, H_0\}$,
\begin{align*}
	\mathcal C(y) \widetilde\chi(y)=e^{iy}\omega_+(y)+e^{-iy}\omega_-(y),\qquad
	|\omega_{\pm}^{(\ell)}(y)|\les \langle y\rangle^{-\frac{1}{2}-\ell}.
\end{align*}

\begin{lemma}\label{G2 lemma}

	Define for $p,q>0$
	\begin{align}\label{G2defn}
		\widetilde G^{\pm}(\lambda,p,q):=\widetilde\chi(\lambda p)\omega_{\pm}(\lambda p)
		-e^{\pm i \lambda (p-  q)}\widetilde\chi(\lambda q)\omega_{\pm}(\lambda q).
	\end{align}
	with $\omega_{\pm}$ as in \eqref{largeJYH}.
	Then for any $0\leq \tau\leq 1$ and $\lambda\leq 2\lambda_1$,
	\begin{align*}
		&|\widetilde G^{\pm}(\lambda,p,q)|\les (\lambda |p-q|)^{\tau} \bigg(\frac{\widetilde\chi(\lambda p)}{|\lambda p|^{\frac{1}{2}}}
		+\frac{\widetilde\chi(\lambda q)}{| \lambda q|^{\frac{1}{2}}}\bigg),\\
		&|\partial_\lambda \widetilde G^{\pm}(\lambda,p,q)|\les |p-q| \bigg(\frac{\widetilde\chi(\lambda p)}{| \lambda p|^{\frac{1}{2}}}
		+\frac{\widetilde\chi(\lambda q)}{|\lambda q|^{\frac{1}{2}}}\bigg)%\frac{\langle x_1\rangle^{\frac{3}{2}}}{\sqrt{\lambda p}}.
	\end{align*}

\end{lemma}

\begin{proof}

	We note first that from \eqref{largeJYH}, we have
	\begin{align}\label{g2triv}
		|\widetilde G^{\pm}(\lambda, p,q)|\les \frac{\widetilde\chi(\lambda p)}{|\lambda p|^{\frac{1}{2}}}
		+\frac{\widetilde\chi(\lambda q)}{| \lambda q|^{\frac{1}{2}}}.
	\end{align}
	We consider the case of $\widetilde G^{+}$, the case of $\widetilde G^-$ is similar.
	Define the function
	\begin{align*}
		b(s):=\widetilde\chi(s)\omega_{+}(s).
	\end{align*}
	 Using \eqref{largeJYH}, one obtains that for $k=0,1,2,...,$
	\begin{align*}
		|b^{(k)}(s)|\lesssim  \widetilde\chi(s)
		|s|^{-\frac{1}{2}-k}.
	\end{align*}
	We now rewrite $\widetilde G$ in terms of $b$:
	\begin{align*}
		\widetilde G(\lambda,p,q)=b(\lambda p)-b(\lambda q)+\big(1-e^{i\lambda(p-q)}\big) b(\lambda q).
	\end{align*}
Note that the absolute value of the last summand is
$$\lesssim \lambda |p-q| \frac{\widetilde\chi(\lambda q)}{|\lambda q|^{\frac{1}{2}}}.$$
To estimate the difference of the first two we assume without loss of generality that $p>q$ and write
$$
|b(\lambda p)-b(\lambda q)|=\Big| \int_{\lambda q}^{\lambda p} b^\prime(s) ds\Big|  \lesssim \int_{\lambda q}^{\lambda p} \widetilde\chi(s)
		|s|^{-\frac{3}{2} } ds.
$$
	 In the case, $1<\lambda q<\lambda p$, we estimate this integral by
$$
\lambda |p-q| \frac{\widetilde\chi(\lambda q)}{|\lambda q|^{\frac{3}{2}}}.
$$
In the case $\lambda q< 1<\lambda p$, we estimate it as follows
\begin{multline}\nn
\int_{\lambda q}^{\lambda p} \widetilde\chi(s)
		|s|^{-\frac{3}{2} } ds\lesssim \widetilde\chi(\lambda p) \int_1^{\lambda p} s^{-3/2} ds
\lesssim \widetilde\chi(\lambda p) \frac{(\lambda p)^{1/2}-1}{|\lambda p|^{1/2}} \leq \\ \widetilde\chi(\lambda p) \frac{(\lambda p)^{1/2}-(\lambda q)^{1/2}}{|\lambda p|^{1/2}} \lesssim  \lambda |p-q| \frac{\widetilde\chi(\lambda p)}{|\lambda p|}.
\end{multline}
Combining these bounds and interpolating with \eqref{g2triv} we obtain the first assertion of the lemma.

	We now turn to the derivative.  We note that
	\begin{align}\nn
		\partial_\lambda \widetilde G(\lambda, p,q)=&pb^\prime(\lambda p)-qb^\prime(\lambda q)+\big(1-e^{i\lambda(p-q)}\big) q b^\prime (\lambda q)-i (p-q)e^{i\lambda(p-q)}b(\lambda q)\\
&=\frac1\lambda [b_1(\lambda p)-b_1(\lambda q)]+\frac{1-e^{i\lambda(p-q)}}{\lambda} b_1(\lambda q)-i (p-q)e^{i\lambda(p-q)}b(\lambda q),\nn
	\end{align}
where $b_1(s):=sb^\prime(s)$ satisfies the same bounds that $b(s)$ does. Therefore the second assertion of the lemma follows as above.
\end{proof}

\begin{lemma}\label{low high log YJ lemma}  We
  have the bound
  \begin{multline}
    \label{low high bound} \bigg| \int_{\R^8} \int_0^\infty e^{it\lambda^2} \lambda \chi(\lambda)\log
    \lambda \chi(\lambda|x-x_1|)Y_0(\lambda|x-x_1|)v(x_1)
     S_1D_1S_1 (x_1,y_1) v(y_1) \\  J_0(\lambda|y-y_1|)
    \widetilde\chi(\lambda|y-y_1|) \, d\lambda f(x) g(y) \, dx_1\, dy_1\, dx\, dy\bigg|\lesssim |t|^{-1}.
  \end{multline}
\end{lemma}

\begin{proof}
	Without loss of generality we assume that $t>0$. As in the proof of the previous statements, it suffices to prove that for fixed $x,x_1,y,y_1$ the  $\lambda$-integral is bounded by $k(x,x_1)\langle y_1\rangle t^{-1}$.  This power of $\langle y_1 \rangle$ necessitates extra decay on
	the potential to push through the $L^2$ mapping bounds as in the previous lemmas.  Accordingly, we assume
	that $v(x)\les \langle x\rangle^{-2-}$ or equivalently that $|V(x)|\les \langle x\rangle^{-4-}$.

Let $p=\max(|y-y_1|,1+|y|)$ and
$q=\min(|y-y_1|,1+|y|)$. Using \eqref{Q trick}, it suffices to consider
\begin{align*}
		\int_0^\infty  e^{it\lambda^2} \lambda \chi(\lambda)\log
    \lambda \,
		F(\lambda, x,x_1) (J_0(\lambda p)
    \widetilde\chi(\lambda p)-J_0(\lambda q)
    \widetilde\chi(\lambda q)) \, d\lambda,
	\end{align*}
	where $F(\lambda, x,x_1)$ is as in Lemma~\ref{lem:FG}. The oscillatory term in
	the definition  \eqref{largeJYH} of $J_0$ for large energies  will move the stationary
	point of the phase. Pulling out the slower oscillation $e^{\pm i\lambda q}$, we rewrite this integral as a sum of
\begin{align*}
		\int_0^\infty  e^{it\phi_\pm(\lambda)} \lambda \chi(\lambda)\log
    \lambda \,
		F(\lambda, x,x_1) \widetilde G^\pm(\lambda,p,q)  \, d\lambda,
	\end{align*}
	where $\phi_{\pm}(\lambda)=\lambda^2\pm
	\lambda q t^{-1}$, and
	$\widetilde G$ is from Lemma~\ref{G2 lemma}.
	Note that this moves the stationary point of the oscillatory integral to $ \lambda_0=\mp\frac{ q}{2t}= \mp\frac{\min(|y-y_1|,1+|y|)}{2t}$.

  We first consider the contribution of the term with the phase  $\phi_-(\lambda)$ in which case the critical point satisfies $\lambda_0\geq 0$.
 Let
	$$a(\lambda):=\lambda \chi(\lambda) \log\lambda F(\lambda,x,x_1) \widetilde G(\lambda,p,q).$$

%%%%%%%%%%%%% % it looks like we don't need to assume this:
% We note that $a$ is supported on the set $1\gtrsim \lambda\gtrsim p^{-1}$, we consider when $\lambda_0$ is in a small
%	neighborhood of the support of $a$.
	
	Using the bounds in Lemma~\ref{lem:FG} and Lemma~\ref{G2 lemma} (with $\tau=0+$), we have
\begin{align}
	|a(\lambda)| & \lesssim k(x,x_1) \langle y_1\rangle^{0+}\lambda \chi(\lambda)
	\bigg(\frac{\widetilde\chi(\lambda p)}{|\lambda p|^{\frac{1}{2}}}
		+\frac{\widetilde\chi(\lambda q)}{| \lambda q|^{\frac{1}{2}}}\bigg),\,\,\,\text{ and }\label{abound23}\\
	|a^\prime(\lambda)| &
\lesssim k(x,x_1)
	\langle y_1\rangle \chi(\lambda) \bigg(\frac{\widetilde\chi(\lambda p)}{|\lambda p|^{\frac{1}{2}}}
		+\frac{\widetilde\chi(\lambda q)}{| \lambda q|^{\frac{1}{2}}}\bigg).\label{aprimebound21}
\end{align}
	We now apply Lemma~\ref{stat phase} with  $a(\lambda)$ as above to bound the $\lambda$-integral in this case by
\begin{align} \label{alemmabound}
		\int_{|\lambda-\lambda_0|<t^{-1/2}}|a(\lambda)|\, d\lambda+ t^{-1}
		\int_{|\lambda-\lambda_0|>t^{-1/2}} \Big( \frac{|a(\lambda)|}{|\lambda-\lambda_0|^2}
	    +\frac{|a'(\lambda)|}{|\lambda-\lambda_0|}\Big)\, d\lambda.
	\end{align}
 Using \eqref{abound23}, we bound the first integral in \eqref{alemmabound} by
	\begin{align}\label{a bound3}
		\int_{|\lambda-\lambda_0|<t^{-1/2}}|a(\lambda)|\, d\lambda
		&\lesssim k(x,x_1) \langle y_1\rangle^{0+} \int_{|\lambda-\lambda_0|<t^{-\frac{1}{2}}}
		\sqrt{\lambda} \bigg(\frac{\widetilde\chi(\lambda p)}{p^{\frac{1}{2}}}
		+\frac{\widetilde\chi(\lambda q)}{ q^{\frac{1}{2}}}\bigg)\, d\lambda.
	\end{align}
There are two cases: $\lambda_0\gtrsim t^{-\frac{1}{2}}$ and $\lambda_0\les t^{-\frac{1}{2}}$. In the former case,
	 on the support of the integral, we have $\lambda\les \lambda_0$. Therefore,
	\begin{align*}
		\eqref{a bound3} &\les k(x,x_1) \langle y_1\rangle^{0+}
		t^{-\frac{1}{2}} \lambda_0^{\frac{1}{2}}(p^{-\frac{1}{2}}+q^{-\frac{1}{2}})
		\les t^{-1} k(x,x_1) \langle y_1\rangle^{0+}.
	\end{align*}
	In the last inequality, we used $p^{-1}\lambda_0\leq q^{-1}\lambda_0 \leq t^{-1}$.
	In the latter case,  on the support of the integral, we have
	$\lambda\les t^{-\frac{1}{2}}$.  So that
	\begin{align*}
		\eqref{a bound3} &\les k(x,x_1) \langle y_1\rangle^{0+}
		\int_{0}^{t^{-\frac{1}{2}}} \sqrt{\lambda}\bigg(\frac{\widetilde\chi(\lambda p)}{ p^{\frac{1}{2}}}
		+\frac{\widetilde\chi(\lambda q)}{ q^{\frac{1}{2}}}\bigg)\,d\lambda
	\end{align*}
	For the $\widetilde\chi(\lambda p)$ term to have any contribution to the integral, we must have that
	$p^{-1}  \les t^{-\frac{1}{2}}$, similarly for $q^{-1}$.  So that,
	\begin{align*}
		\eqref{a bound3}&\les k(x,x_1) \langle y_1\rangle^{0+} (p^{-\frac{1}{2}} +q^{-\frac{1}{2}})t^{-\frac{3}{4}}
		\les k(x,x_1) \langle y_1\rangle^{0+} t^{-1}.
	\end{align*}
	
It suffices to bound the second integral in \eqref{alemmabound} by $k(x,x_1)\langle y_1\rangle$.
%	For the contribution of the second summand in Lemma~\ref{stat phase}, we need only show that
%	\begin{align*}
%	    \int_{|\lambda-\lambda_0|>|t|^{-1/2}} \frac{|a(\lambda)|}{|\lambda-\lambda_0|^2}
%	    +\frac{|a'(\lambda)|}{|\lambda-\lambda_0|}\, d\lambda
%	\end{align*}
%	is bounded.	
%	We again consider the size of the critical point of the phase $\lambda_0$.
	We first establish the bounds for the $a(\lambda)$ term and then consider the derivative
	$a^\prime(\lambda)$.

We have two cases: $\lambda_0\ll t^{-\frac{1}{2}}$ and $\lambda_0\gtrsim t^{-\frac{1}{2}}$. In the former case, we have $|\lambda-\lambda_0|\approx \lambda$. Thus, using \eqref{abound23}, we obtain
	\begin{align}
		\int_{|\lambda-\lambda_0|>t^{-1/2}} \frac{|a(\lambda)|}{|\lambda-\lambda_0|^2} \, d\lambda
\les k(x,x_1) \langle y_1\rangle^{0+}
		\int_{\R} \lambda^{-\frac{3}{2}} \bigg(\frac{\widetilde\chi(\lambda p)}{p^{\frac{1}{2}}}
		+\frac{\widetilde\chi(\lambda q)}{ q^{\frac{1}{2}}}\bigg)\, d\lambda
\les k(x,x_1) \langle y_1\rangle^{0+}.\nn
	\end{align}
	
In the latter case,   we have
\begin{align*}
	    \int_{|\lambda-\lambda_0|>t^{-1/2}} \frac{|a(\lambda)|}{|\lambda-\lambda_0|^2}
	    \, d\lambda
	    &\les k(x,x_1) \langle y_1\rangle^{0+}
	    \int_{|\lambda-\lambda_0|>t^{-\frac{1}{2}}}\frac{\lambda^{\frac{1}{2}}}{|\lambda-\lambda_0|^2}
	     \bigg(\frac{\widetilde\chi(\lambda p)}{p^{\frac{1}{2}}}
		+\frac{\widetilde\chi(\lambda q)}{ q^{\frac{1}{2}}}\bigg)\,d\lambda
\end{align*}
Changing the variable $s=\lambda-\lambda_0$ and recalling that $p\geq q$, we bound this   by
	\begin{align*}
		k(x,x_1) \langle y_1\rangle^{0+} q^{-\frac{1}{2}}
	    \int_{|s|>t^{-\frac{1}{2}}}\frac{s^{\frac{1}{2}}+\lambda_0^{\frac12}}{s^2}\,ds
	  \les k(x,x_1) \langle y_1\rangle^{0+} q^{-\frac{1}{2}} (t^{1/4}+t^{1/2}\lambda_0^{1/2}) \les k(x,x_1) \langle y_1\rangle^{0+}.
	\end{align*}
The last inequality follows from the assumption  $t^{-1/2}\les \lambda_0=\frac{q}{2t}$.
	
	Now, we consider the contribution of $a^\prime(\lambda)$. Again we have two cases: $\lambda_0\ll t^{-\frac{1}{2}}$ and $\lambda_0\gtrsim t^{-\frac{1}{2}}$. In the former case, we have $|\lambda-\lambda_0|\approx \lambda$. Thus, using \eqref{aprimebound21}, we obtain
\begin{align*}
	   \int_{|\lambda-\lambda_0|>t^{-1/2}} \frac{|a'(\lambda)|}{|\lambda-\lambda_0|}\, d\lambda&\les k(x,x_1) \langle y_1 \rangle \int_{\R}\lambda^{-\frac{3}{2}}
	    \bigg(\frac{\widetilde\chi(\lambda p)}{p^{\frac{1}{2}}}
		+\frac{\widetilde\chi(\lambda q)}{ q^{\frac{1}{2}}}\bigg)\, d\lambda \les k(x,x_1) \langle y_1 \rangle.
	\end{align*}

In the latter case, we have
\begin{multline*}
	      \int_{|\lambda-\lambda_0|>t^{-1/2}} \frac{|a'(\lambda)|}{|\lambda-\lambda_0|}\, d\lambda
	    \les k(x,x_1)\langle y_1\rangle \int_{|\lambda-\lambda_0|>t^{-\frac{1}{2}}}
	    \frac{1}{|\lambda-\lambda_0|\lambda^\frac12 }\bigg(\frac{\widetilde\chi(\lambda p)}{ p^{\frac{1}{2}}}
		+\frac{\widetilde\chi(\lambda q)}{q^{\frac{1}{2}}}\bigg)\, d\lambda\\
 \lesssim  k(x,x_1)\langle y_1\rangle \Big[ q^{-\frac12}\int_{|\lambda-\lambda_0|>t^{-\frac{1}{2}}}  \frac{d\lambda }{|\lambda-\lambda_0|^{\frac32} }+ \int_\R \frac{d\lambda }{\lambda^{\frac32} }
	    \bigg(\frac{\widetilde\chi(\lambda p)}{p^{\frac{1}{2}}}
		+\frac{\widetilde\chi(\lambda q)}{ q^{\frac{1}{2}}}\bigg)\Big] \\
  \les k(x,x_1)\langle y_1\rangle \big[q^{-\frac12}t^{1/4} +1\big]\les  k(x,x_1)\langle y_1\rangle.
\end{multline*}
The second inequality follows from $\frac1{|\lambda-\lambda_0|	\lambda^{1/2}}\leq \frac1{|\lambda-\lambda_0|^{3/2}}+\frac1{ 	\lambda^{3/2}}	$, and the last one from the assumption  $t^{-1/2}\les \lambda_0=\frac{q}{2t}$.

	When considering the phase $\phi_{+}(\lambda)=\lambda^2+
	\lambda q t^{-1}$, integration by parts suffices to obtain the desired bound since the phase has no critical points on $(0,\infty)$. We have
	\begin{align*}
		\bigg|\int_0^\infty e^{it\phi_+(\lambda)}a(\lambda)\, d\lambda\bigg|
		&\les t^{-1}\bigg( \int_0^\infty \frac{|a(\lambda)|}{|\phi_+'(\lambda)|^2}\,d\lambda
		+\int_0^\infty \frac{|a'(\lambda)|}{|\phi_+'(\lambda)|}\, d\lambda\bigg).
\end{align*}
Using \eqref{abound23},  \eqref{aprimebound21}, and     $\phi_{+}^\prime(\lambda)\geq 2\lambda$, we bound the right hand side by
\begin{align*}
&\les k(x,x_1)\langle y_1\rangle t^{-1} \int_0^\infty \lambda^{-\frac32}  \bigg(
	 \frac{\widetilde\chi(\lambda p)}{ p^{\frac{1}{2}}}
		+\frac{\widetilde\chi(\lambda q)}{  q^{\frac{1}{2}}}\bigg) \, d\lambda \les  k(x,x_1)\langle y_1\rangle t^{-1}. \qedhere
\end{align*}
\end{proof}

	When switching roles of $\widetilde{\chi}$ and $\chi$ in \eqref{low high bound}, we note that
	from \eqref{largeJYH} the high energy Bessel function representation holds for
	$Y_0(y)$ as well.  The proof will move along   the same line, with $G(\lambda,y,y_1)$ in place
	of $\widetilde G(\lambda,y,y_1)$ and
	using $\widetilde G(\lambda,x,x_1)$ in place of $F(\lambda, x,x_1)$.
	The case when both Bessel functions are $Y_0$ or $J_0$ is similar:

\begin{corollary}\label{low high YJ lemma2}

	For $\mathcal C(z)=J_0(z)$ or $\mathcal C(z)=Y_0(z)$, we have the bound
	\begin{multline}    		\label{low high bound2}
    		\bigg|  \int_{\R^8} \int_0^\infty e^{it\lambda^2} \lambda \chi(\lambda)\chi(\lambda|x-x_1|)
		\mathcal C(\lambda|x-x_1|)v(x_1) S_1D_1S_1 (x_1,y_1) v(y_1) \\
		 \mathcal C(\lambda|y-y_1|)
    		\widetilde\chi(\lambda|y-y_1|) \, d\lambda f(x) g(y) \, dx_1\, dy_1\, dx\, dy\bigg|\les |t|^{-1}.
  	\end{multline}
\end{corollary}

We now consider the case when both Bessel functions are supported on high energies. For this we will use the first line of \eqref{K defn}:
\begin{lemma}\label{high high log YJ lemma}

  We have the bound
  \begin{multline} \label{high high bound}
    \bigg| \int_{\R^8} \int_0^\infty e^{it\lambda^2} \lambda \chi(\lambda)h^\pm(\lambda)  \widetilde\chi(\lambda|x-x_1|)
    H^\pm_0(\lambda|x-x_1|)v(x_1)
    S_1D_1S_1 (x_1,y_1) v(y_1)\\   H^\pm_0(\lambda|y-y_1|)
    \widetilde\chi(\lambda|y-y_1|) \, d\lambda f(x) g(y) \, dx_1\, dy_1\, dx\, dy\bigg|\les |t|^{-1}.
  \end{multline}
\end{lemma}

\begin{proof}

Again we assume that $t>0$. Recall \eqref{JYasymp2}:
	\begin{align}\label{Hankel exp high}
		H_0^{\pm}(z)\widetilde\chi(z)=e^{\pm iz}\omega_{\pm}(z)
	\end{align}
	with  $|\omega_{\pm}^{(\ell)}(z)|\lesssim (1+|z|)^{-\frac{1}{2}-\ell}$.
 %Further, as $H_0^+=\overline{H_0^-}$, we see that $\omega_-=\overline{\omega_+}$.
As in Lemma~\ref{low high log YJ lemma} we need
	to use the auxiliary function $\widetilde G$.  Denote $p_1=\max(|x-x_1|,1+|x|)$,
	$p_2=\min(|x-x_1|,1+|x|)$, $q_1=\max(|y-y_1|,1+|y|)$ and $q_2=\min(|y-y_1|,1+|y|)$.
	Without loss of generality, $p_1,p_2,q_1,q_2>0$.
 We note that by \eqref{Q trick},
	we can replace $H_0^\pm(\lambda |x-x_1|)$ with $\widetilde G(\lambda,p_1,p_2)$ (and similarly replace $H_0^\pm(\lambda |y-y_1|)$) as in Lemma~\ref{G2 lemma}. This changes the phase with
\begin{align*}
		\phi_{\pm}(\lambda)=\lambda^2\pm \lambda \frac{p_2+q_2}{t}.
\end{align*}
We first consider
	$\phi_{-}(\lambda)$, which has a stationary point at $\lambda_0=\frac{p_2+q_2}{2t}>0$.
We  apply  Lemma~\ref{stat phase} with
	$$a(\lambda)=\lambda \chi(\lambda)h^{\pm}(\lambda) \widetilde G(\lambda ,p_1,p_2)\widetilde G(\lambda, q_1,q_2).$$
Using the bounds in Lemma~\ref{G2 lemma} (with $\tau=0+$ for $\widetilde G(\lambda ,p_1,p_2)$ and $\tau=0$ for the other), we have
\begin{align}
	|a(\lambda)| &\lesssim \langle x_1\rangle^{0+}  \chi(\lambda)
	\bigg(\frac{\widetilde\chi(\lambda p_1)}{ p_1^{1/2}}
		+\frac{\widetilde\chi(\lambda p_2)}{p_2^{1/2}}\bigg)\bigg(\frac{\widetilde\chi(\lambda q_1)}{ q_1^{1/2}}
		+\frac{\widetilde\chi(\lambda q_2)}{q_2^{1/2}}\bigg),\,\,\,\,\,\text{ and }\label{newabound23} \\
	|a^\prime(\lambda)|& 
\lesssim  \langle x_1\rangle
	\langle y_1\rangle \frac{\chi(\lambda)}{\lambda} \bigg(\frac{\widetilde\chi(\lambda p_1)}{ p_1^{1/2}}
		+\frac{\widetilde\chi(\lambda p_2)}{p_2^{1/2}}\bigg)\bigg(\frac{\widetilde\chi(\lambda q_1)}{ q_1^{1/2}}
		+\frac{\widetilde\chi(\lambda q_2)}{q_2^{1/2}}\bigg).\label{newaprimebound21}
\end{align}
Let $\tau=2\max(\lambda_0,t^{-1/2})$. Since $\tilde\chi$ is a nondecreasing function supported on $[1,\infty)$ and $\lambda\leq \tau$ on the support of the integral, we have the bound
\begin{multline}\label{ilksira}
		\int_{|\lambda-\lambda_0|<t^{-1/2}} |a(\lambda)|\, d\lambda  \lesssim \langle x_1\rangle^{0+}  t^{-1/2}
\bigg(\frac{\widetilde\chi(\tau p_1)}{ p_1^{1/2}}
		+\frac{\widetilde\chi(\tau p_2)}{p_2^{1/2}}\bigg)\bigg(\frac{\widetilde\chi(\tau q_1)}{ q_1^{1/2}}
		+\frac{\widetilde\chi(\tau q_2)}{q_2^{1/2}}\bigg) \\
		 \lesssim \langle x_1\rangle^{0+}  t^{-1/2} \tau \lesssim  \langle x_1\rangle^{0+}   t^{-1}, 
	\end{multline}
if $\tau=2t^{-1/2}$. On the other hand if $\tau=2\lambda_0$, we consider the contributions of the products of $\tilde\chi$'s more carefully. Consider the contribution of
\begin{align*}
		\frac{\widetilde\chi(\tau p_i)\widetilde\chi(\tau q_j)}{p_i^{1/2}q_j^{1/2}}
\end{align*}
to \eqref{ilksira}. This term is zero unless $p_i\gtrsim 1/\lambda_0$ and $q_j\gtrsim 1/\lambda_0$. Therefore,
using  $p_1\geq p_2$ and $q_1\geq q_2$, we have
\begin{align}\label{piqjcalc}
	\frac{p_2+q_2}{p_iq_j}\leq \frac{p_i+q_j}{p_i q_j}=\frac{1}{p_i}+\frac{1}{q_j}\les   \lambda_0=	\frac{p_2+q_2}{2t}.
\end{align}
Therefore, $p_iq_j\gtrsim t$, and we can estimate the contribution of  each product to \eqref{ilksira} by $\langle x_1\rangle^{0+}  t^{-1}$.

	For the portion of $a(\lambda)$ supported on $|\lambda-\lambda_0|>t^{-\frac{1}{2}}$, we note that
	if $\lambda_0\ll t^{-\frac{1}{2}}$, then $|\lambda-\lambda_0|\approx \lambda$ so that
	\begin{multline*}
		\int_{|\lambda-\lambda_0|>t^{-\frac{1}{2}}} \frac{|a(\lambda)|}{|\lambda-\lambda_0|^2}\, d\lambda
		 \les \langle x_1\rangle^{0+} \sum_{i,j=1}^2 \int_{\R}   \frac{\widetilde\chi(\lambda p_i)}{ p_i^{1/2}}
		 \frac{\widetilde\chi(\lambda q_j)}{ q_j^{1/2}}
		 \, \frac{d\lambda}{\lambda^2}\\
		 \les \langle x_1\rangle^{0+} \sum_{i,j=1}^2 \bigg( \int_\R \frac{\widetilde \chi(\lambda p_i)^2}{p_i }\,\frac{d\lambda}{\lambda^2} \bigg)^{\frac{1}{2}}
		\bigg(\int_{\R} \frac{\widetilde\chi(\lambda q_j)^2}{q_j }\,\frac{d\lambda}{\lambda^2} \bigg)^{\frac{1}{2}}\les \langle x_1\rangle^{0+}.
	\end{multline*}
	On the other hand if $\lambda_0 \gtrsim t^{-\frac{1}{2}}$ , we have
	\begin{align*}
		\int_{|\lambda-\lambda_0|>t^{-\frac{1}{2}}} \frac{|a(\lambda)|}{|\lambda-\lambda_0|^2}\, d\lambda
		&\les \langle x_1\rangle^{0+} \sum_{i,j=1}^2 (p_iq_j)^{-\frac{1}{2}} \int_{|\lambda-\lambda_0|>t^{-\frac{1}{2}}}
		\frac{\widetilde\chi(\lambda p_i)  \widetilde\chi(\lambda q_j)}{|\lambda-\lambda_0|^2}\, d\lambda.
	\end{align*}
	Fix $i,j$ and let $m=\min(p_i,q_j)$. We have two cases: $\lambda_0\ll 1/m$ and $\lambda_0\gtrsim 1/m$. In the former case, we note that $|\lambda-\lambda_0|\gtrsim 1/m$ on the support of the cutoffs. Therefore,
\begin{align*}(p_iq_j)^{-\frac{1}{2}} \int_{|\lambda-\lambda_0|>t^{-\frac{1}{2}}}
		\frac{\widetilde\chi(\lambda p_i)  \widetilde\chi(\lambda q_j)}{|\lambda-\lambda_0|^2}\, d\lambda
\les (p_iq_j)^{-\frac{1}{2}} \int_{|\lambda-\lambda_0|\gtrsim 1/m}
		\frac{d\lambda}{|\lambda-\lambda_0|^2}  \les (p_iq_j)^{-\frac{1}{2}} m \leq 1.
\end{align*}
In the latter case, using \eqref{piqjcalc}, we conclude that $p_iq_j\gtrsim t$. This implies the desired bound by ignoring the cutoffs in the integral.

	We now turn to the term in Lemma~\ref{stat phase} that involves $a'(\lambda)$.
Using \eqref{newaprimebound21}, we have 	
	\begin{align*}\nn %\label{aprime bound}
		&\int_{|\lambda-\lambda_0|>t^{-1/2}}\frac{|a^\prime(\lambda)|}{|\lambda-\lambda_0|} \,d\lambda   \les \langle x_1\rangle \langle y_1\rangle \sum_{i,j=1}^2 (p_iq_j)^{-\frac{1}{2}}
		\int_{|\lambda-\lambda_0|>t^{-1/2}}
		\frac{ \widetilde\chi(\lambda p_i)  \widetilde\chi(\lambda q_j)}{\lambda|\lambda-\lambda_0|}d\lambda\\
& \les \langle x_1\rangle \langle y_1\rangle \sum_{i,j=1}^2 (p_iq_j)^{-\frac{1}{2}}
		\Big[\int_{|\lambda-\lambda_0|>t^{-1/2}}
		\frac{ \widetilde\chi(\lambda p_i)  \widetilde\chi(\lambda q_j)}{\lambda^2}d\lambda
+\int_{|\lambda-\lambda_0|>t^{-1/2}}
		\frac{ \widetilde\chi(\lambda p_i)  \widetilde\chi(\lambda q_j)}{ |\lambda-\lambda_0|^2}d\lambda\Big].
	\end{align*}
	The required bounds for each of these terms appeared above in the bound for $a(\lambda) /|\lambda-\lambda_0|^2$ integral.
	This establishes the desired bound for the phase $\phi_{-}$.  For the case of $\phi_{+}$, integration by parts and
	the bounds on $a(\lambda)$ and $a'(\lambda)$ suffice, we leave the details to the reader.		
\end{proof}

With these estimates established, we are ready to prove Proposition~\ref{prop:summand2}.

\begin{proof}[Proof of Proposition~\ref{prop:summand2}]

	Lemmas~\ref{low log YJ lemma}, \ref{low high log YJ lemma},
	\ref{high high log YJ lemma} and Corollaries~\ref{low YJ lemma2}, \ref{low high YJ lemma2} bound each term of  \eqref{K defn} as desired.
\end{proof}

We now turn to the terms involving $SS_1D_1S_1$ and $S_1D_1 S_1 S$ in Corollary~\ref{M exp cor1}.
\begin{prop}\label{S and S_1 terms lemma} The contribution of the  terms $SS_1D_1S_1$,  $S_1D_1S_1S$ and $QD_0Q$ in Corollary~\ref{M exp cor1} in \eqref{klambda} satisfies \eqref{disp}. More explicitly,
we have the bound
  	\begin{multline} \nn
    		\bigg| \int_{\R^8} \int_0^\infty e^{it\lambda^2} \lambda \chi(\lambda)
		R_0^{\pm}(\lambda^2)(x,x_1)v(x_1)
    		SS_1D_1S_1 (x_1,y_1) v(y_1) \\
		R_0^{\pm}(\lambda^2)(y,y_1) \, d\lambda f(x) g(y) \, dx_1\, dy_1\, dx\, dy\bigg|\les |t|^{-1}.
%    	\label{high high bound}	
  	\end{multline}
  	The same bound holds when $SS_1D_1S_1$ is replaced by $S_1D_1S_1S$ or by $QD_0Q$.
\end{prop}

\begin{proof}
	The $QD_0Q$ term can be  handled as in Proposition~\ref{prop:summand2}, it is in fact easier since there is no $\log(\lambda)$ term.

	The other terms are somehow different since they have a projection orthogonal to $v$ only on one side. Therefore, one can   use \eqref{Q trick} only on one side.  However, since there is no $\log(\lambda)$ term, the bounds established in   in Lemmas~\ref{low log YJ lemma}, \ref{low high log YJ lemma},
	and~\ref{high high log YJ lemma} go through.    For instance, to establish the bound
	\begin{multline}\nn
    		\bigg| \int_{\R^8} \int_0^\infty e^{it\lambda^2} \lambda \chi(\lambda)\chi(\lambda|x-x_1|)Y_0(\lambda|x-x_1|)v(x_1)
    		SS_1 D_1S_1(x_1,y_1) v(y_1)  \\  J_0(\lambda|y-y_1|)
    		\chi(\lambda|y-y_1|) \, d\lambda f(x) g(y) \, dx_1\, dy_1\, dx\, dy\bigg|\les |t|^{-1},
  	\end{multline}
	we can use $G(\lambda,y,y_1)$ in place of $J_0(\lambda|y-y_1|)$. After an integration by parts the boundary terms   vanish since  $G(\lambda,y,y_1)\to 0$ as $\lambda\to 0$, and the $\lambda$-integral can be bounded by
\begin{multline}\nn
    		 t^{-1} \int_0^\infty     \Big|\partial_\lambda\Big(\chi(\lambda)\chi(\lambda|x-x_1|)Y_0(\lambda|x-x_1|)G(\lambda,y,y_1)\Big)\big| \, d\lambda  \les \\
t^{-1} \langle y_1\rangle^{\tau} (1+\log^{-}(|x-x_1|)) \int_0^1  (1+|\log(\lambda)|)\lambda^{\tau-1} d\lambda \les  t^{-1} \langle y_1\rangle^{\tau} (1+\log^{-}(|x-x_1|)).
  	\end{multline}
Here we used the bounds for $G$ from Lemma~\ref{lem:FG}, the bounds \eqref{Y0 def2} and \eqref{Jprime}, and the following estimate:
\be \label{chichilog}
 \chi(\lambda) \chi(\lambda|x-x_1|) \log(\lambda|x-x_1|) \les (1+|\log(\lambda)|)(1+\log^{-}(|x-x_1|)).
\ee
This estimate follows easily by considering the cases $|x-x_1|<1$ and $|x-x_1|>1$ separately.

	When we have $S_1D_1S_1 S$ instead, we must use $F(\lambda,x,x_1)$
	instead of $Y_0(\lambda|x-x_1|)$, and the boundary terms are now controlled by $|t|^{-1}k(x,x_1)$ as
	in Lemma~\ref{lem:FG}.  The other cases when both projections are onto low energies  can be handled similarly.

The case when $\chi$ is replaced with $\widetilde\chi$ on both sides can be handled as in Lemma~15 from \cite{Sc2} since the argument there does not make use of the projections orthogonal to $v$.

Similarly, in the case when $\chi$ is replaced with $\widetilde \chi$ on the side which does not have a projection orthogonal to $v$, the proof of Lemma~14 from \cite{Sc2} applies.

It remains to prove that
\begin{multline}\label{nolog}
    		\bigg| \int_{\R^8} \int_0^\infty e^{it\lambda^2} \lambda \chi(\lambda)\chi(\lambda|x-x_1|)Y_0(\lambda|x-x_1|)v(x_1)
    		SS_1 D_1S_1(x_1,y_1) v(y_1)  \\  J_0(\lambda|y-y_1|)
    		\widetilde \chi(\lambda|y-y_1|) \, d\lambda f(x) g(y) \, dx_1\, dy_1\, dx\, dy\bigg|\les |t|^{-1}.
\end{multline}
Since $S$ is not orthogonal to $v$, we can not replace $\chi Y_0$ with $F$.   However, we can replace  $\widetilde\chi J_0$ with $\widetilde G$ shifting the critical point of the $\lambda$-integral as in the proof of Lemma~\ref{low high log YJ lemma}.
The argument in the proof of that lemma relies on the bounds
\be\label{log F bounds}
|\log(\lambda) F(\lambda,x,x_1)|\les \log(\lambda) k(x,x_1),\,\,\,\,\,\,\,
\big|\partial_\lambda\big(\log(\lambda) F(\lambda,x,x_1)\big)\big|\les k(x,x_1) \log(\lambda) \lambda^{-1}.
\ee
Since we don't have an additional $\log(\lambda)$ in \eqref{nolog}, it suffices to note that  $\chi(\lambda) \chi(\lambda|x-x_1|)Y_0(\lambda|x-x_1|)$ satisfies similar bounds  as in \eqref{log F bounds} (with $1+|\log^-(|x-x_1)|$ instead of $k(x,x_1)$, {\em c.f.}, \eqref{chichilog}).
\end{proof}

  The terms arising from $h_{\pm}(\lambda)^{-1}S$ and $h_{\pm}(\lambda)^{-1}SS_1D_1S_1S$
are handled in Lemma~17 in \cite{Sc2}, which we restate below for completeness.

\begin{prop}\cite{Sc2} \label{S prop} The contribution of the  terms $h_{\pm}(\lambda)^{-1}S$ and $h_{\pm}(\lambda)^{-1}SS_1D_1S_1S$ in Corollary~\ref{M exp cor1} in \eqref{klambda} satisfies \eqref{disp}. More explicitly,
	we have the bound
	\begin{align*}
		 \bigg|\int_0^\infty e^{it\lambda^2}\lambda \chi(\lambda) \bigg\langle \bigg[
		\frac{R_0^{+}(\lambda^2)vSvR_0^{+}(\lambda^2)}{h_+(\lambda)}
		-\frac{R_0^{-}(\lambda^2)vSvR_0^{-}(\lambda^2)}{h_-(\lambda)}
		\bigg]f,g\bigg\rangle \, d\lambda \bigg|
		 \lesssim |t|^{-1}.%\label{S bound}
	\end{align*}
A similar bound holds if we replace $S$ with $SS_1D_1S_1S$.
\end{prop}

%\begin{proof}

%	 In \cite{Sc2} it was shown that
%	 \begin{align*}
%	 	\bigg|&\int_0^\infty e^{it\lambda^2} \lambda \chi(\lambda)\bigg[\frac{H_0^+(\lambda|x-x_1|)H_0^+(\lambda|y_1-y|)}{h_+(\lambda)}
%		 -\frac{H_0^-(\lambda|x-x_1|)H_0^-(\lambda|y_1-y|)}{h_-(\lambda)}\bigg]\,d\lambda\bigg|\nn \\
%		&\lesssim |t|^{-1}(1+\log^-|x-x_1|)(1+\log^-|y-y_1|).
%	 \end{align*}
%	 Then an analysis as in \eqref{1st der bound} finishes the proof.
%\end{proof}

Finally the following proposition (Lemma~18 from \cite{Sc2}) takes care of the contribution of the error
term in Corollary~\ref{M exp cor1} to \eqref{klambda}.
%We restate this result in a slightly different form due to our
%error size estimates in the expansion of $M^{-1}(\lambda)$.  For the previous lemmas we have assumed
%$v(x)\les \langle x\rangle^{-(1+k)-}$ for $k=\frac{5}{2}$.  Thus, we also obtain that the error term in
%Corollary~\ref{M exp cor1} satisfies $\mathcal O_1(\lambda^{\frac{1}{2}})$.
\begin{prop}\cite{Sc2} \label{Error term lemma}
	Assume that $\Phi(\lambda)$ is an absolutely bounded operator on $L^2(\R^2)$ that satisfies
	$\Phi(\lambda)=\mathcal O_1(\lambda^{\frac{1}{2}})$.  We have the bound
  	\begin{multline}\nn%\label{error O bound}
    		\bigg| \int_{\R^8}  \int_0^\infty e^{it\lambda^2} \lambda \chi(\lambda)
		R_0^{\pm}(\lambda^2)(x,x_1)v(x_1)
    		\Phi(\lambda) (x_1,y_1)v(y_1)
		 R_0^{\pm}(\lambda^2)(y,y_1) \, d\lambda  \\
f(x) g(y) \, dx_1\, dy_1\, dx\, dy\bigg|\les |t|^{-1}.
  	\end{multline}
\end{prop}

%%%%%%%%%%%%%%%%%%%%%%%%%%%%%%%%%%%%%%%%%%%%%%%%
%%%%%%%%%%%%%%%%%%%%%%%%%%%%%%%%%%%%%%%%%%%%%%%
%%%%%%%%%%%%%%%%%%%%%%%%%%%%%%%%%%%%%%%%%%%%%%%
%%%%%%%%%%%%%%%%%%%%%%%%%%%%%%%%%%%%%%%%%%%%%%%
%%%%%%%%%%%%%%%%%%%%%%%%%%%%%%%%%%%%%%%%%%%%%%%
%%%%%%%%%%%%%%%%%%%%%%%%%%%%%%%%%%%%%%%%%%%%%%%
%%%%%%%%%%%%%%%%%%%%%%%%%%%%%%%%%%%%%%%%%%%%%%%
%%%%%%%%%%%%%%%%%%%%%%%%%%%%%%%%%%%%%%%%%%%%%%%
%%%%%%%%%%%%%%%%%%%%%%%%%%%%%%%%%%%%%%%%%%%%%%%
%%%%%%%%%%%%%%%%%%%%%%%%%%%%%%%%%%%%%%%%%%%%%%%
%%%%%%%%%%%%%%%%%%%%%%%%%%%%%%%%%%%%%%%%%%%%%%%
%%%%%%%%%%%%%%%%%%%%%%%%%%%%%%%%%%%%%%%%%%%%%%%
%%%%%%%%%%%%%%%%%%%%%%%%%%%%%%%%%%%%%%%%%%%%%%%
%%%%%%%%%%%%%%%%%%%%%%%%%%%%%%%%%%%%%%%%%%%%%%%
%%%%%%%%%%%%%%%%%%%%%%%%%%%%%%%%%%%%%%%%%%%%%%%
%%%%%%%%%%%%%%%%%%%%%%%%%%%%%%%%%%%%%%%%%%%%%%%
%%%%%%%%%%%%%%%%%%%%%%%%%%%%%%%%%%%%%%%%%%%%%%
%%%%%%%%%%%%%%%%%%%%%%%%%%%%%%%%%%%%%%%%%%%%%%%%
%%%%%%%%%%%%%%%%%%%%%%%%%%%%%%%%%%%%%%%%%%%%%%%
%%%%%%%%%%%%%%%%%%%%%%%%%%%%%%%%%%%%%%%%%%%%%%%
%%%%%%%%%%%%%%%%%%%%%%%%%%%%%%%%%%%%%%%%%%%%%%%
%%%%%%%%%%%%%%%%%%%%%%%%%%%%%%%%%%%%%%%%%%%%%%%
%%%%%%%%%%%%%%%%%%%%%%%%%%%%%%%%%%%%%%%%%%%%%%%
%%%%%%%%%%%%%%%%%%%%%%%%%%%%%%%%%%%%%%%%%%%%%%%
%%%%%%%%%%%%%%%%%%%%%%%%%%%%%%%%%%%%%%%%%%%%%%%
%%%%%%%%%%%%%%%%%%%%%%%%%%%%%%%%%%%%%%%%%%%%%%%
%%%%%%%%%%%%%%%%%%%%%%%%%%%%%%%%%%%%%%%%%%%%%%%
%%%%%%%%%%%%%%%%%%%%%%%%%%%%%%%%%%%%%%%%%%%%%%%
%%%%%%%%%%%%%%%%%%%%%%%%%%%%%%%%%%%%%%%%%%%%%%%
%%%%%%%%%%%%%%%%%%%%%%%%%%%%%%%%%%%%%%%%%%%%%%%
%%%%%%%%%%%%%%%%%%%%%%%%%%%%%%%%%%%%%%%%%%%%%%%
%%%%%%%%%%%%%%%%%%%%%%%%%%%%%%%%%%%%%%%%%%%%%%%
%%%%%%%%%%%%%%%%%%%%%%%%%%%%%%%%%%%%%%%%%%%%%%%
%%%%%%%%%%%%%%%%%%%%%%%%%%%%%%%%%%%%%%%%%%%%%%

\section{Resonances of the second and third kind}\label{sec:pwave}

We now consider the evolution in the case of a p-wave
resonance and/or an eigenvalue at zero. Recall that this case is characterized by the non-invertibility of  $T_1=S_1TPTS_1$.
 To obtain resolvent expansions around zero, we need to invert the operator $B_\pm$, \eqref{B defn}. The expansions in this section are  considerably more complicated
than those in the case of a resonance of the first kind given in Proposition~\ref{prop:Binverse1}.
 %and as such are slightly more opaque than the corresponding results in the s-wave only case.

Recall the operators $S_2, S_3, T_2,$ and $T_3$ from  Definition~\ref{resondef}. With a slight abuse of the notation, we define $D_1:=(T_1+S_2)^{-1}=(S_1TPTS_1+S_2)^{-1}$
as an operator on $S_1L^2(\R^2)$.
We define $D_2:=(T_2+S_3)^{-1}=(S_2vG_1vS_2+S_3)^{-1}$ on $S_2L^2(\R^2)$, we will also use $D_2$ for $T_2^{-1}$ when $T_2$ is invertible, i.e. when $S_2=0$. We also define $D_3:=T_3^{-1}=(S_3vG_2vS_3)^{-1}$ on $S_3 L^2(\R^2)$.

\begin{prop}\label{prop:Binverse2}
 Assume that $v(x)\les \langle x\rangle^{-3-}$. Then, $B_{\pm}$ is invertible on
$S_1L^2(\R^2)$. In the case of a resonance of the second kind,  we have
	\begin{align}\label{secondk}
		B^{-1}_{\pm}=\frac{S_2D_2S_2}{g_1^\pm(\lambda)}+O(\lambda^{-2}|\log\lambda|^{-2}),
	\end{align}
where  $g_{1}^{\pm}(\lambda)$ is as in Lemma~\ref{lem:M_exp}.
%	(B_{1,\pm}+S_2)^{-1}&=[T_1+S_2 -h_{\pm}(\lambda) E^{\pm}(\lambda)]^{-1},\,\,\,\text{ and}\\
%	B_{2,\pm}^{-1}&=-h_{\pm}^{-1}(\lambda)g_{1}^{\pm}(\lambda)^{-1}\big[T_2+\lambda^2g_{1}^{\pm}(\lambda)^{-1}
%	S_2vG_2vS_2+E_2^{\pm}(\lambda)\big]^{-1},
%  \end{align*}
%  with $E^{\pm}(\lambda)=\mathcal O_1(\lambda^{2-})$, $E_2^{\pm}(\lambda)=\mathcal O_1(\lambda^{2-})$,  $h_{\pm}(\lambda)$ is as in Lemma~\ref{M+S1inverse}, and $g_{1}^{\pm}(\lambda)$ is as in Lemma~\ref{lem:M_exp}, and $T_2$ is invertible on $S_2L^2$.

In the case of a resonance of the third kind,  we have
\begin{align}\label{thirdk}
		B^{-1}_{\pm}=\frac{S_3D_3S_3}{\lambda^2 }+\frac{S_2\mathcal D S_2}{ g_1^\pm(\lambda)} +O(\lambda^{-2}|\log\lambda|^{-2}),
	\end{align}
Here $\mathcal D=D_2+S_3D_3S_3vG_2vS_2D_2S_2vG_2vS_3D_3S_3-S_3D_3S_3vG_2vS_2D_2-D_2S_2vG_2vS_3D_3S_3$.
%where $\widetilde D_2= D_3S_3vG_2vS_2D_2S_2vG_2vS_3D_3$.
\end{prop}

\begin{proof}
	We give the proof for the case of the ``$+$" superscripts and subscripts and omit them from the proof.
    Recall the definition \eqref{B defn} of $B$:
$$ B=S_1-S_1(M(\lambda)+S_1)^{-1}S_1.$$
First we repeat the expansion that we obtained in Proposition~\ref{prop:Binverse1} by keeping   track of the error term better.
 Using Lemma~\ref{M+S1inverse}, the identity
    $$A^{-1}(\lambda)\big[\mathbbm{1}+M_0A^{-1}(\lambda)\big]^{-1}=A^{-1}(\lambda)-
    A^{-1}(\lambda)M_0A^{-1}(\lambda)\big[\mathbbm{1}+M_0A^{-1}(\lambda)\big]^{-1},$$
    and the definition \eqref{Ainverse} of $A^{-1}(\lambda)$, we obtain
	\begin{align}\label{bbb}
		B=S_1-S_1(h(\lambda)^{-1}S+QD_0Q)S_1+ E(\lambda),
	\end{align}
where
\begin{align}\label{E defn}
	&E(\lambda)=S_1A^{-1}(\lambda)M_0(\lambda)A^{-1}(\lambda)\big[\mathbbm{1}
	+M_0(\lambda)A^{-1}(\lambda)\big]^{-1}S_1\\
	&=S_1A^{-1}(\lambda)M_0(\lambda)A^{-1}(\lambda)S_1
%\nn \\&
-S_1A^{-1}(\lambda)
	[M_0(\lambda)A^{-1}(\lambda)]^2\big[\mathbbm{1}+M_0(\lambda)A^{-1}(\lambda)\big]^{-1}S_1.\nn
\end{align}
Since  $v(x)\les \langle x\rangle^{-3-}$, by Lemma~\ref{lem:M_exp}, we have $M_0= \mathcal O_1(\lambda^{2-})$. Also using $A^{-1}(\lambda)=\mathcal O_1(1)$ (from \eqref{Ainverse}), we conclude that
\be\label{elambda}
E(\lambda)=\mathcal O_1(\lambda^{2-}).
\ee
	
Recall that $S_1 D_0=D_0 S_1=S_1$.
	Further, from the definition \eqref{S_defn} of $S$,  and the fact that $S_1P=PS_1=0$, we obtain
$S_1SS_1=S_1 TPT S_1$. Therefore
	\begin{align}\label{newBdefn}
		B=-h(\lambda)^{-1}S_1SS_1+E(\lambda) =-h(\lambda)^{-1} S_1 TPT S_1+ E(\lambda).
	\end{align}
In the case of a resonance of the second kind (unlike the case of a resonance of the first kind), the leading term $T_1=S_1TPTS_1$ above is not invertible. We will invert the operator
\begin{align*}
	B_1:= -h(\lambda)B= T_1 -h(\lambda) E(\lambda),
\end{align*}
by using Lemma~\ref{JNlemma}. Let $S_2$ be the Riesz projection onto the kernel of $T_1$, and let $  D_1:=(T_1+S_2)^{-1}$. We have
\begin{align}
	(B_1+S_2)^{-1}&=[T_1+S_2 -h(\lambda) E(\lambda)]^{-1}
	=  D_1\big[\mathbbm{1}-h(\lambda) E(\lambda)  D_1\big]^{-1}\nn \\
	&=  D_1+ D_1 h(\lambda) E(\lambda) D_1 + D_1 \big[h(\lambda) E(\lambda)  D_1\big]^2
	\big[\mathbbm{1}-h(\lambda) E(\lambda)  D_1\big]^{-1}
 =D_1+ O(\lambda^{2-}). \label{b1s2-11}
\end{align}
By Lemma~\ref{JNlemma}, $ B_1$ is invertible if
\begin{align*}
	B_2:=S_2-S_2(B_1+S_2)^{-1}S_2
\end{align*}
is invertible on $S_2L^2$.
Using \eqref{b1s2-11}, the identities $S_2D_1=D_1S_2=S_2$, and the definition \eqref{E defn} of $E(\lambda)$, we have
\begin{align*}
	B_2 & =-h(\lambda) S_2E(\lambda)S_2 - S_2\big[h(\lambda) E(\lambda)D_1\big]^2
	\big[\mathbbm{1}-h(\lambda) E(\lambda)D_1\big]^{-1}S_2\\
	&=-h(\lambda) S_2A^{-1}(\lambda)M_0(\lambda)A^{-1}(\lambda)S_2+E_1(\lambda),
\end{align*}
where
\begin{multline}\label{E1 defn}
	E_1(\lambda) =  h(\lambda) S_2A^{-1}(\lambda)[M_0(\lambda)A^{-1}(\lambda)]^2\big[\mathbbm{1}
	+M_0(\lambda)A^{-1}(\lambda)\big]^{-1} S_2\\
	 - S_2\big[h(\lambda) E(\lambda)D_1\big]^2\big[\mathbbm{1}-h(\lambda) E(\lambda)D_1\big]^{-1}S_2.%\nn
\end{multline}

 We now claim that
\be \label{pts2}
PTS_2=S_2TP=0.
\ee  To see this, note that   since $S_2$, $T$ and $P$ are self-adjoint, and $S_2$ is the projection onto the kernel of $S_1TPTS_1$, we have
\begin{align*}
	\langle PTS_2 f, PTS_2 f\rangle=\langle S_2TPTS_2 f, f\rangle=
	\langle S_2S_1TPTS_1S_2 f, f\rangle=0.
\end{align*}
 Therefore,
\be\label{A-1S2} A^{-1}(\lambda)S_2=S_2A^{-1}(\lambda)=S_2.
\ee
Using this and the expansion \eqref{M0_defn} of $M_0$, we rewrite $B_2$ as
\begin{align*}
B_2&=-h(\lambda)g_1(\lambda)S_2vG_1vS_2-h(\lambda)\lambda^2 S_2vG_2vS_2-h(\lambda)S_2M_1(\lambda)S_2+E_1(\lambda)\\
&=:-h(\lambda)g_1(\lambda)\big[T_2+ \lambda^2g_1^{-1}(\lambda) S_2vG_2vS_2+E_2(\lambda)\big].
\end{align*}
By Lemma~\ref{lem:M_exp}, we have $M_0= \mathcal O_1(\lambda^{2-})$ and $h(\lambda)S_2M_1(\lambda)S_2=\mathcal O_1(\lambda^{2-})$. Also using $A^{-1}(\lambda)=\mathcal O_1(1)$ and \eqref{elambda}, we conclude that
 $E_1(\lambda)=\mathcal O_1(\lambda^{4-})$. This yields that $E_2(\lambda)=\mathcal O_1(\lambda^{2-})$.  In the case of a resonance of the second kind the leading term is invertible.
Therefore, for small $\lambda$,
\begin{align}
B_2^{-1} =-\frac{\big[T_2+\lambda^2g_1^{-1}(\lambda)S_2vG_2vS_2+E_2(\lambda)\big]^{-1}}{h(\lambda)g_1(\lambda)} =-\frac{D_2}{h(\lambda)g_1(\lambda)} +O(\lambda^{-2} |\log \lambda|^{-3}).\label{b2-1}
\end{align}
Using Lemma~\ref{JNlemma}, \eqref{b1s2-11},  \eqref{b2-1}, and the identities $S_2D_1=D_1S_2=S_2$, we have
\begin{align}\label{b-1b2form}
B^{-1}&=-h(\lambda)B_1^{-1}=-h(\lambda)\big[(B_1+S_2)^{-1}+(B_1+S_2)^{-1}S_2B_2^{-1}S_2(B_1+S_2)^{-1}\big]\\
&= \frac{S_2D_2S_2}{g_1(\lambda)}
+O((\lambda \log \lambda)^{-2}).\nn
\end{align}

In the case of a resonance of the third kind, the leading term in $B_2$ is not invertible. Analogously, we will invert the   operator
\begin{align*}
B_3 =-h^{-1}(\lambda)g_1^{-1}(\lambda)B_2 =T_2+\lambda^2g_1^{-1}(\lambda)S_2vG_2vS_2+E_2(\lambda)
\end{align*}
by using Lemma~\ref{JNlemma}. Let $S_3$ be the Riesz projection onto the kernel of $T_2$, and let $D_2:=(T_2+S_3)^{-1}$. We have
\begin{align}
\nn	(B_3+&S_3)^{-1}=\big[T_2+S_3+\lambda^2g_1^{-1}(\lambda)S_2vG_2vS_2+E_2(\lambda)\big]^{-1}\\
%\nn	&=D_2\big[\mathbbm{1}+\lambda^2g_1^{-1}(\lambda)S_2vG_2vS_2D_2+E_2(\lambda)\big]^{-1}\\
\label{b3s31}	&=D_2- \lambda^2 g_1^{-1}(\lambda)D_2 S_2vG_2vS_2D_2+ \lambda^4g_1^{-2}(\lambda) D_2[S_2vG_2vS_2D_2]^2+O(|\log \lambda|^{-3}).
%\\ \label{b3s32} &= D_2+O(|\log \lambda|^{-1}) .
\end{align}
In the second line we used the definition of $g_1(\lambda)$ in
Lemma~\ref{lem:M_exp} and the estimate on $E_2(\lambda)$.

By Lemma~\ref{JNlemma}, $ B_3$ is invertible if
\begin{align*}
	B_4:=S_3-S_3(B_3+S_3)^{-1}S_3
\end{align*}
is invertible on $S_3L^2$.
Using \eqref{b3s31}, the identities $S_3D_2=D_2S_3=S_3$, and $T_3=S_3vG_2vS_3$,  we have
\begin{align*}
	B_4&=\lambda^2g_1^{-1}(\lambda) T_3- \lambda^4g_1^{-2}(\lambda) S_3 [S_2vG_2vS_2D_2]^2S_3+O(|\log \lambda|^{-3}).
\end{align*}
Since $T_3$ is always invertible (see Section~4 of \cite{JN}), $B_4$ is invertible for small $\lambda$, and we have
\begin{align*}
	B_4^{-1}&=\lambda^{-2}g_1(\lambda) D_3+\widetilde D_2 + O(|\log\lambda|^{-1}).
\end{align*}
%&=:\lambda^2g_1^{-1}(\lambda)T_3\big[\mathbbm{1} + \lambda^2g_1^{-1}(\lambda) D_3S_3[S_2vG_2vS_2D_2]^2S_3+ O((\log\lambda)^{-2})\big].
where $\widetilde D_2= D_3S_3vG_2vS_2D_2S_2vG_2vS_3D_3$.

Using this, Lemma~\ref{JNlemma}, and \eqref{b3s31},    we have
\begin{align*}
&B_2^{-1} =-\frac1{h(\lambda)g_1(\lambda)}B_3^{-1}=-\frac1{h(\lambda)g_1(\lambda)}\big[(B_3+S_3)^{-1}+(B_3+S_3)^{-1}S_3B_4^{-1}S_3(B_3+S_3)^{-1}\big]\\
&=-\frac{S_3D_3S_3}{\lambda^2 h(\lambda)}-\frac{D_2+S_3\widetilde D_2S_3-S_3D_3S_3vG_2vS_2D_2-D_2S_2vG_2vS_3D_3S_3}{h(\lambda)g_1(\lambda)} +O(\lambda^{-2}|\log\lambda|^{-3}).
\end{align*}
Using this (instead of \eqref{b2-1}) for $B^{-1}$ in \eqref{b-1b2form}  yields the assertion of the proposition.
\end{proof}

\begin{corollary}\label{M exp cor2}

	Assume that $v(x)\les \langle x\rangle^{-3-}$.  Then,
	in the case of a resonance of the second kind, we have
	\begin{multline}\label{M pwave}
		M^{\pm}(\lambda)^{-1} =\frac{S_2D_2S_2}{g_1^{\pm}(\lambda)}
		+Q\Gamma_{1}^{\pm}(\lambda)Q+Q\Gamma_2^{\pm}(\lambda) + \Gamma_3^{\pm}(\lambda)Q+
		 \Gamma_4^{\pm}(\lambda) \\
		  +(M^{\pm}(\lambda)+S_1)^{-1}+\mathcal O_1(\lambda^{2-}),
	\end{multline}
	where
	$\Gamma_{i}^{\pm}$, $i=1,2,3,4$ are absolutely bounded operators on $L^2(\R^2)$ with
	$\Gamma_{1}^{\pm}(\lambda)=O(\lambda^{-2}( \log \lambda)^{-2})$,
	$\Gamma_{2}^{\pm}(\lambda), \Gamma_3^{\pm}(\lambda)=O(\lambda^{-2}( \log \lambda)^{-3})$, and
	$\Gamma_{4}^{\pm}(\lambda)=O(\lambda^{-2}( \log \lambda)^{-4})$.

 In the case of a resonance of the third kind, we have
	\begin{multline}\label{M eval}
		M^{\pm}(\lambda)^{-1} =\frac{S_3D_3S_3}{\lambda^2 }+\frac{S_2\mathcal D S_2}{ g_1^\pm(\lambda)}
		+Q\Gamma_{1}^{\pm}(\lambda)Q+Q\Gamma_2^{\pm}(\lambda) + \Gamma_3^{\pm}(\lambda)Q+
		 \Gamma_4^{\pm}(\lambda) \\
		+(M^{\pm}(\lambda)+S_1)^{-1}+\mathcal O_1(\lambda^{2-}),
	\end{multline}
	where $\mathcal D$ is as in Proposition~\ref{prop:Binverse2}, and $\Gamma_i$ are absolutely bounded operators on $L^2(\R^2)$.  These operators are distinct from the
	$\Gamma_i $ in the case of a resonance of the second kind, but satisfy the same size estimates.

\end{corollary}

\begin{proof}
For a resonance of the second kind, combining Proposition~\ref{prop:Binverse2} with Lemma~\ref{JNlemma} and Lemma~\ref{M+S1inverse} (taking the decay condition on $v$ into account),  we have
\begin{align*}
		M(\lambda)^{-1}&=(M(\lambda)+S_1)^{-1}+(M(\lambda)+S_1)^{-1}S_1B^{-1}S_1
		(M(\lambda)+S_1)^{-1}\\
		&=(M(\lambda)+S_1)^{-1}
% \\ &
+\frac1{g_1(\lambda)}\big(A(\lambda)^{-1}  +\mathcal O_1(\lambda^{2-})\big) S_2D_2
		S_2 \big(A(\lambda)^{-1}  +\mathcal O_1(\lambda^{2-})\big) \\
&+ \big(A(\lambda)^{-1}  +\mathcal O_1(\lambda^{2-})\big) O((\lambda \log \lambda)^{-2})  \big(A(\lambda)^{-1}  +\mathcal O_1(\lambda^{2-}))\big).
\end{align*}
Using \eqref{A-1S2} and the definition \eqref{Ainverse} of $A^{-1}$, we obtain
\begin{align*}
		M(\lambda)^{-1}&=\frac{S_2D_2S_2}{g_1(\lambda)} + (M(\lambda)+S_1)^{-1}+\mathcal O_1(\lambda^{2-})  \\
&+ \big(QD_0Q  +O(|\log \lambda|^{-1}) \big) O((\lambda \log \lambda)^{-2})  \big(QD_0Q  +O(|\log \lambda|^{-1})\big).
\end{align*}
The second line leads to four different terms yielding \eqref{M pwave}.

For the case of a resonance of the third kind, the statement follows similarly using the formula \eqref{thirdk} for $B^{-1}$.
\end{proof}

We now consider the dispersive estimates in the case  when $H$ has  a  p-wave  resonance  at zero energy.
Comparing \eqref{M pwave} to the expansion in Corollary~\ref{M exp cor1},
we note that the many of the terms in the expansion for resonances
of the second kind are in the expansion  for resonances of the first kind. Accordingly, it suffices to establish the estimates for the contributions of the terms:
\begin{align*}
  \frac{S_2D_2S_2}{g_1^{\pm}(\lambda)}
		+Q\Gamma_{1}(\lambda)Q+Q\Gamma_2(\lambda) + \Gamma_3(\lambda)Q+
		\Gamma_4(\lambda).
\end{align*}
We start with the following.

\begin{lemma}\label{D2 lemma}

	We have the bound
  	\begin{align}\label{D2 int bound}
    		\bigg| \int_0^\infty e^{it\lambda^2} \lambda \chi(\lambda)\bigg\langle\bigg[
		\frac{R_0^{+}(\lambda^2)v S_2 D_2S_2 v R_0^{+}(\lambda^2)}{g_1^{+}(\lambda)}
    		-\frac{R_0^{-}(\lambda^2) v D_2  v R_0^{-}(\lambda^2)}{g_1^{-}(\lambda)}\bigg]
		f, g\bigg\rangle\, d\lambda \bigg|
		&\les 1.
  	\end{align}

\end{lemma}

\begin{proof}

  We note that we must exploit some cancellation between the `$+$' and `$-$' terms.  Recall that
  $H_0^{\pm}(y)=J_0(y)\pm iY_0(y)$ and the definition of $g_1^{\pm}(\lambda)$ in Lemma~\ref{lem:M_exp} give us
  \begin{align}
  \label{pwave hankel cancel}  \frac{R_0^{+}(\lambda^2) R_0^{+}(\lambda^2)}{g_1^{+}(\lambda)}
    -\frac{R_0^{-}(\lambda^2)  R_0^{-}(\lambda^2)}{g_1^{-}(\lambda)}
    &=\frac{J_0(\lambda p)J_0(\lambda q)-Y_0(\lambda p)Y_0(\lambda q)}{\lambda^2[(\log \lambda +c_1)^2+c_2^2]} \\
    &+\frac{(J_0(\lambda p)Y_0(\lambda q)+Y_0(\lambda p)J_0(\lambda q)) (\log \lambda +c_1)}
    {\lambda^2[(\log \lambda +c_1)^2+c_2^2]}\nn
  \end{align}
    We again must use the cut-offs $\chi$ and $\widetilde\chi$ and consider the different cases depending
  the supports of the resolvents.  Let us first consider the case when both resolvents are supported on low energy.
  Contribution of the first term in \eqref{pwave hankel cancel} satisfies the required bound since $J_0=O(1)$, and $\frac{1}{\lambda(\log\lambda)^2}$ is integrable on $[0,\lambda_1]$. Since the other terms have additional powers $\log\lambda$ in the numerator, we need to use \eqref{Q trick} (recall that $S_2 \leq Q$).

  Consider the contribution of the second term  in \eqref{pwave hankel cancel}. Using \eqref{Q trick}, we replace $\chi Y_0$ with $F(\lambda,\cdot,\cdot)$, and using Lemma~\ref{lem:FG}, we obtain the bound:
 $$
        \bigg|\int_0^\infty \chi(\lambda) \frac{F(\lambda, x, x_1) F(\lambda,y,y_1)}
      {\lambda [(\log \lambda +c_1)^2+c_2^2]}\, d\lambda \bigg|
      \lesssim k(x,x_1)k(y,y_1) \int_0^1 \frac{1}
      {\lambda (\log \lambda)^2}\, d\lambda
      \lesssim k(x,x_1)k(y,y_1).
$$
  The mixed $J_0$ and $Y_0$ terms in the second part
  of \eqref{pwave hankel cancel} are bounded similarly using $|G(\lambda,x,x_1)|\les \lambda^{0+}\langle x_1\rangle^{0+}$.

  An analysis as in \eqref{1st der bound} shows that these terms satisfy the desired bound \eqref{D2 int bound}.

When one or both of the Bessel functions is supported on high energies, we use the functions $\widetilde G(\lambda,p,q)$ from Lemma~\ref{G2 lemma}.
The bound $|\widetilde G(\lambda,p,q)|\les \lambda^{0+}|p-q|^{0+}$ suffices for obtaining the required bound.
  The details are left to the reader.
\end{proof}

\begin{lemma}\label{low high YJ lemma3}

	For $\mathcal C_i(z)=J_0(z)$ or $Y_0(z)$ for $i=1,2$, we have the bound
	$$
    		\bigg| \int_{\R^8} \int_0^\infty e^{it\lambda^2} \lambda \chi(\lambda)
		\mathcal C_1(\lambda|x-x_1|)v(x_1) Q \Gamma_1(\lambda) Q (x_1,y_1) v(y_1)
		\mathcal C_2(\lambda|y-y_1|) \, d\lambda f(x) g(y) \, dx_1\, dy_1\, dx\, dy
		\bigg|\les 1.
  	$$

\end{lemma}

\begin{proof}

	Unlike in Lemma~\ref{D2 lemma}
	we do not need to use any cancellation between the `$+$' and `$-$' terms.  We consider the terms that arise when both
	$\mathcal C_1$ and $\mathcal C_2$ are supported on small energies.  Consider,
	\begin{align*}
		\int_0^\infty e^{it\lambda^2}\lambda \chi(\lambda) \chi(\lambda p)\mathcal C_1(\lambda p)
		vQ\Gamma_1(\lambda)Qv \chi(\lambda q)\mathcal C_2(\lambda q)
		\, d\lambda,
	\end{align*}
	where $p=|x-x_1|$, $q=|y-y_1|$. In the worst case when $\mathcal C_1=\mathcal C_2=Y_0$, using \eqref{Q trick}, we replace  $\chi Y_0 $ with $F $   to obtain
\begin{multline}
\label{FF Ci low}
		\bigg|\int_0^1 \lambda  F(\lambda,x,x_1) \Gamma_1(\lambda)F(\lambda,y,y_1)\, d\lambda\bigg|
		 \lesssim \bigg| \int_0^1
		 \frac{F(\lambda,x,x_1)F(\lambda,y,y_1)}{\lambda (\log\lambda)^2}\, d\lambda\bigg|
		 \sup_{0<\lambda<\lambda_1}|\lambda^2 (\log \lambda)^2 \Gamma_1(\lambda)|%\nn\\
		 \\ \lesssim k(x,x_1) k(y,y_1) \sup_{0<\lambda<\lambda_1}|\lambda^2 (\log \lambda)^2 \Gamma_1(\lambda)|.%\nn
	\end{multline}
	The last line follows from Lemma~\ref{lem:FG}. Since
  $\sup_{0<\lambda<\lambda_1} |\lambda^2 (\log \lambda)^2\Gamma_1(\lambda)|$
	defines a bounded operator on $L^2(\R^2)$ (by Corollary~\ref{M exp cor2}), we are done.
	The other low energy terms are similar using $G$ instead of $F$ from Lemma~\ref{lem:FG}.

	For the large energies, we note that the argument runs in a similar manner.  Using
	$\widetilde\chi(y) (|J_0(y)|+|Y_0(y)|)\lesssim 1$, and an argument as in \eqref{FF Ci low}, it easily follows that  the integral is bounded
	as desired.
\end{proof}

The following modification of Lemma~\ref{low high YJ lemma3} is necessary for the other $\Gamma_{i} (\lambda)$
terms.

\begin{corollary}\label{low high YJ cor3}

	For $\mathcal C_i(z)=J_0(z)$ or $Y_0(z)$ for $i=1,2$, we have the bound
	$$
    		\bigg|  \int_{\R^8} \int_0^\infty e^{it\lambda^2} \lambda \chi(\lambda)
		\mathcal C_1(\lambda|x-x_1|)v(x_1) Q\Gamma_2(\lambda)(x_1,y_1)   v(y_1)
\mathcal C_2(\lambda|y-y_1|) \, d\lambda  \, dx_1\, dy_1 \bigg|\les 1.$$

  The same bounds hold when $Q\Gamma_2(\lambda) $ is replaced by
  	$ \Gamma_3(\lambda)Q$ or $ \Gamma_4(\lambda $.

\end{corollary}

\begin{proof}
	We repeat the analysis of Lemma~\ref{low high YJ lemma3}.
 Consider the case when both $\mathcal C_i(\lambda \cdot)$  are supported on
	low energies and both are $Y_0$.    We note that when $\lambda<1$, using \eqref{chichilog}, we have
	\begin{align}\label{Y0 low bd}
		|Y_0(\lambda p)\chi(\lambda p)| \lesssim (1+|\log\lambda|)(1+\log^- p).
	\end{align}
	Using this and replacing $\chi Y_0$ with $F$ on one side, we obtain the bound
	$$
      	  \int_0^{\lambda_1} \frac{|F(\lambda, x,x_1)| (1+\log^{-} q)}{\lambda |\log \lambda|^2}\, d\lambda
		    \sup_{0<\lambda<\lambda_1} |\lambda^2 (\log \lambda)^3 \Gamma_2(\lambda)|
      		 \lesssim  k(x,x_1) k(y,y_1)
		 \sup_{0<\lambda<\lambda_1} |\lambda^2 (\log \lambda)^3 \Gamma_2(\lambda)|.
  	$$
	The same bound holds for  $\Gamma_3(\lambda)Q$.  For the contribution of $ \Gamma_4(\lambda) $, we have
	\begin{align*}
      		&\bigg|\int_0^\infty \lambda \chi(\lambda) Y_0(\lambda p)  \Gamma_4(\lambda)  Y_0(\lambda q)
      		\, d\lambda \bigg| \\
      		&\lesssim  \int_0^1 \frac{(1+|\log \lambda|)(1+\log^{-} p)(1+|\log \lambda|)(1+\log^{-} q)}
		{\lambda |\log \lambda|^4}\, d\lambda   \sup_{0<\lambda<\lambda_1} |\lambda^2 (\log \lambda)^4\Gamma_4(\lambda)|\\
      		&\lesssim k(x,x_1) k(y,y_1)\sup_{0<\lambda<\lambda_1} |\lambda^2 (\log \lambda)^4\Gamma_4(\lambda)|.
  	\end{align*}
  The other cases are similar.

When one of the $\mathcal C_i(\lambda \cdot)$ is supported
	on high energies, the analysis is less delicate. The required bound follows from $\widetilde\chi(y)(|J_0(y)|+|Y_0(y)|)\les 1$.
\end{proof}
This completes the proof in the case of a resonance of the second kind.

We note that the above bounds in Lemma~\ref{low high YJ lemma3} and Corollary~\ref{low high YJ cor3}
also hold for the $\Gamma_i$ term in \eqref{M eval}.  Thus for a resonance of the third kind,
it suffices to consider the leading $\lambda^{-2}$ term in \eqref{M eval}.  Noting \eqref{K defn} and
the fact that the kernel of $D_3$ is real-valued, the following lemma completes the proof. We will prove in the next section that $G_0vS_3D_3S_3vG_0$ is the projection onto the zero eigenspace whose contribution disappears since we project away from the zero eigenspace. We will ignore this issue in the proof below since the eigenfunctions are bounded functions and hence the projection onto the zero eigenspace satisfies the desired bound, and since removing this operator requires more decay from the potential, see Section~\ref{sec:weighted}.
\begin{lemma}\label{D3 lemma}

	We have the bound
	\begin{align}\label{low high bound6}
    		\bigg| &\int_{\R^8} \int_0^\infty e^{it\lambda^2} \lambda \chi(\lambda)
		J_0(\lambda|x-x_1|)v(x_1) \frac{S_3D_3S_3 }{\lambda^2}  v(y_1)  Y_0(\lambda|y-y_1|) \, d\lambda   \, dx_1\, dy_1 \bigg|\les 1.
  	\end{align}
\end{lemma}

\begin{proof}

	We provide a sketch of the proof.  Due to similarities to previous proofs, we leave the details to the reader.
	We again consider the case when the Bessel functions are supported on low energy first.  Accordingly,
	we wish to control
	\begin{align*}
		\bigg|\int_0^\infty e^{it\lambda^2}\lambda &\chi(\lambda) \chi(\lambda p) J_0(\lambda p)v(x_1)\frac{D_3}{\lambda^2}
		v(y_1) \chi(\lambda q) Y_0(\lambda q) \, d\lambda\bigg|\\
		&\lesssim \bigg| \int_0^\infty \chi(\lambda) \frac{G(\lambda,x,x_1)F(\lambda,y,y_1)}{\lambda}\, d\lambda
		\lesssim \langle x_1\rangle^{\tau}k(y,y_1).
	\end{align*}
	Where we used \eqref{Q trick}, Lemma~\ref{lem:FG} with any $\tau>0$.
	
	For the case when one function is supported on high energy, we have
	\begin{align*}
		\bigg|\int_0^\infty e^{it\lambda^2}\lambda &\chi(\lambda) \widetilde\chi(\lambda p)
		J_0(\lambda p)v(x_1)\frac{D_3}{\lambda^2}
		v(y_1) \chi(\lambda q) Y_0(\lambda q) \, d\lambda\bigg|\\
		&\lesssim \bigg| \int_0^\infty \chi(\lambda) \frac{\widetilde G(\lambda,x,x_1)F(\lambda,y,y_1)}{\lambda}\, d\lambda
		\lesssim \langle x_1\rangle^{\tau}k(y,y_1).
	\end{align*}
	Similarly one uses $\widetilde G(\lambda,y,y_1)$ instead of $F(\lambda, y,y_1)$ if we have $\widetilde\chi(\lambda q)$.
	
	When both functions are supported on high energy, we have
	\begin{align*}
		\bigg|\int_0^\infty e^{it\lambda^2}\lambda &\chi(\lambda) \widetilde\chi(\lambda p)
		J_0(\lambda p)v(x_1)\frac{D_3}{\lambda^2}
		v(y_1) \widetilde\chi(\lambda q) Y_0(\lambda q) \, d\lambda\bigg|\\
		&\lesssim \bigg| \int_0^\infty \chi(\lambda) \frac{\widetilde G(\lambda,x,x_1)\widetilde G(\lambda,y,y_1)}{\lambda}\, d\lambda
		\lesssim \langle x_1\rangle^{\tau}\langle y_1\rangle^{\tau}.
	\end{align*}
	An analysis as in \eqref{1st der bound} finishes the proof.
\end{proof}

We are now ready to prove the main theorem of this section.

\begin{theorem}\label{pwave thm}

	Let $V:\R^2\to \R$ be such that $|V(x)|\lesssim \langle x\rangle^{-\beta}$ for some $\beta>6$.  Further assume that
	$H=-\Delta+V$ has a resonance of the second or third kind at zero energy.  Then, there is a time dependent operator
	$F_t$ such that
	\begin{align*}
		\sup_t \| F_t\|_{L^1\to L^\infty}\lesssim 1,
		\qquad \|K_{\lambda_1} - F_t \|_{L^1\to L^\infty}\lesssim |t|^{-1},\,\,\,\,|t|>1.
	\end{align*}

\end{theorem}

\begin{proof}
If we denote the terms that arise from the contribution of the terms in the first lines of  \eqref{M pwave} and \eqref{M eval} as $F_t$,
Lemmas~\ref{D2 lemma}, \ref{low high YJ lemma3}, and~\ref{D3 lemma}
and Corollary~\ref{low high YJ cor3} show that
\begin{align*}
	\sup_t \| F_t\|_{L^1\to L^\infty}\lesssim 1.
\end{align*}

As the remaining terms in \eqref{M pwave} and \eqref{M eval} are identical in form to those that arise in the proof of Theorem~\ref{swave thm}, we can use the bounds from the previous section to   establish the theorem.
\end{proof}

Finally, we note that Theorem~\ref{big thm} follows directly from
Theorem~\ref{high eng thm}, Theorem~\ref{swave thm}, Theorem~\ref{pwave thm}, and the first remark following Definition~\ref{resondef}.

\section{Spectral Structure of $-\Delta+V$ at Zero Energy}\label{sec:spectral}
In this section, we prove some of the claims made in the remark following Definition~\ref{resondef}.  In particular, we show the relationship between the
spectral subspaces $S_i L^2(\R^2)$ for $1=1,2,3$ and distributional solutions to $H\psi =0$.

Let $w=Uv$.
First we characterize $S_1L^2$. \\
\begin{lemma}\label{S1 characterization}
 	If $v\les \la x\ra^{-1-}$ and if $\phi\in S_1L^2$, then $\phi = w \psi$ where $\psi\in L^\infty$, $H\psi=0$
 	in the sense of distributions, and
	$$
		\psi=c_0-G_0v\phi,\,\,\,\,\,c_0=\frac1{\|V\|_{L^1}}\la v,T\phi \ra.
	$$
	Moreover, if $v\les\la x\ra^{-2-}$, then $\psi-c_0=-G_0v\phi\in L^p$ for any $p\in(2,\infty]$.
\end{lemma}

\begin{proof}
Since $\phi\in S_1L^2$, we have $Q\phi =\phi$.  Using this and $P=I-Q$, we have
$$
0=QTQ\phi=QT\phi=T\phi-PT\phi=U\phi+vG_0v\phi-PT\phi.
$$
Thus,
$$
\phi=-wG_0v\phi+UPT\phi=-wG_0v\phi+wc_0=w\psi.
$$
Also note that since $v(x)\lesssim \la x\ra^{-1-}$ and $\phi\in L^2$, we have  $-\Delta G_0(v\phi)=v\phi$.
Therefore, we see that $H\psi=0$ by taking the distributional derivative.

Now we prove that $\psi\in L^\infty$. The boundedness on $B(0,4)$ is clear. To see that $\psi$ is bounded for $|x|>4$, use $P\phi=0$ to obtain
$$
G_0v\phi(x)=-\frac1{2\pi}\int[\log(|x-y|)-\log(|x|)]v(y)\phi(y)dy=-\frac1{2\pi}\int \log\Big(\frac{|x-y|}{|x|}\Big) v(y)\phi(y)dy.
$$
The bound follows by using the inequality (for $|x|>4$)
\begin{align}	\label{logtrick}
	\Big|\log\Big(\frac{|x-y|}{|x|}\Big)\big|\les  1+\log(\la y\ra)+\log^-(|x-y|).
\end{align}
Note that this only requires that  $v(x)\lesssim \la x\ra^{-1-}$.

The final statement follows if we can prove that $G_0v\phi=O(|x|^{-1})$ for large $x$. To see this, write
\begin{align*}
G_0v\phi(x) &=-\frac1{2\pi}\int \log\Big(\frac{|x-y|}{|x|}\Big) v(y)\phi(y)dy\\
&=\int_{|y|>|x|/2}\log\Big(\frac{|x-y|}{|x|}\Big) v(y)\phi(y)dy+\int_{|y|<|x|/2}\log\Big(\frac{|x-y|}{|x|}\Big) v(y)\phi(y)dy.
\end{align*}
The first integral can be estimated by
$$
 \int_{|y|>|x|/2} [1+\log(\la y\ra)+\log^-(|x-y|)] \frac{|y| |v(y)\phi(y)|}{|x|}dy=O(1/|x|).
$$
On the other hand, the bound for the second integral follows from
$$
\Big|\log\Big(\frac{|x-y|}{|x|}\Big)\big|= \Big|\log\Big(1+\frac{|x-y|-|x|}{|x|}\Big)\big|
=O\Big(\frac{ ||x-y|-|x| |}{|x|} \Big)=O(|y|/|x|).
$$
\end{proof}

\begin{lemma}\label{Hpsi characterization}
 	Let $v\les \la x\ra^{-2-}$. Assume that the function $\psi=c+\Lambda_1+\Lambda_2$, with $\Lambda_1\in L^p$, for some $p\in(2,\infty)$,
	and $\Lambda_2\in L^2$, solves $H\psi=0$ in the sense of distributions. Then $\phi=w\psi\in S_1L^2$, and we have $\psi=c-G_0v\phi$, $c=\frac1{\|V\|_{L^1}}\la v,T\phi \ra$. In particular,
	by the previous claim, $\psi-c\in L^p$ for any $p\in(2,\infty]$.
\end{lemma}

\begin{proof}
Since $H\psi=0$, we have
$$
\Delta \psi=V\psi=v\phi.
$$
This easily implies that $\int v(y)\phi(y)dy =0$, see \cite[Lemma 6.4]{JN}.
Thus $\phi\in QL^2$.

Now consider the function
$\psi+G_0v\phi.$
By the calculation above, we see that $\Delta(\psi+G_0v\phi)=0$. Since  $\psi+G_0v\phi\in L^2+L^\infty$ (by assumption and the proof of the previous claim), we see that it has to be a constant. Thus
$$
\psi=c-G_0v\phi.
$$
Using this, we have
$$
TQ\phi=T\phi=U\phi+vG_0v\phi=U\phi-v\psi+cv= U\phi- U\phi+cv=cv,
$$
and hence $QTQ\phi=0$, and $\phi\in S_1L^2$. Finally, this implies that $c=\frac1{\|V\|_{L^1}}\la v,T\phi \ra$.
\end{proof}

Note that Lemma~\ref{S1 characterization} and Lemma~\ref{Hpsi characterization} imply that all zero eigenfunctions are bounded. We now characterize $S_2L^2$.

\begin{lemma}\label{S2 characterization}
Assume that $|v(x)|\les\la x\ra^{-3-}$. Then $\phi =w\psi \in S_2L^2$ if and only if
$ \psi \in L^p$, for all $p\in(2,\infty]$ (or equivalently $c_0=0$).
\end{lemma}

\begin{proof}
Recall that $S_2\leq S_1$ is projection onto the kernel of $S_1TPTS_1$. We have (since $S_1\phi=\phi$)
\be\label{ptphi}
S_1TPTS_1\phi=0  \implies \|PT\phi\|^2=\la TPT\phi,\phi \ra=0,
\ee
and hence $c_0 =0$ and $\psi\in L^p$ for $p>2$.

On the other hand if $\psi\in L^p$ for $p>2$, we have $c_0=0$. This implies that $PT\phi=PTS_1\phi=0$ ,and hence $S_1TPTS_1\phi=0$.
\end{proof}

\begin{lemma}\label{G2 invert}
If $v\les \la x\ra^{-3-}$ then the kernel of the operator $S_3vG_2vS_3$ on $S_3L^2$ is trivial.
\end{lemma}

\begin{proof}
Given $f$ in the kernel of $S_3vG_2vS_3$,  we have
\be \label{tempor}
\int_{\R^2} f(y)v(y)\, dy =0, \text{ and } S_2vG_1vS_2 f=0,
\ee
since $f\in S_3L^2\subset QL^2$.

Also note that the expansion we used for $R_0^+(\lambda^2)$ in the proof of Lemma~\ref{lem:M_exp} gives that
$$
R_0^+(\lambda^2)(x,y)=g^+(\lambda)+G_0(x,y)+g_1^+(\lambda)G_1(x,y)+  \lambda^2 G_2(x,y)
+O(\lambda^{2+}|x-y|^{2+}).
$$
This and the assumption $v\les \la x\ra^{-3-}$ imply that
\begin{align*}
	0&= \la S_3vG_2vS_3 f,f\ra = \la vG_2 vf, f\ra
	= \lim_{\lambda\to 0}\bigg\la \frac{R_0^+(\lambda^2)-g^+(\lambda)-G_0-g_1^+(\lambda)G_1}
	{\lambda^2}vf, vf\bigg\ra.
\end{align*}
Now, using \eqref{tempor},   we have
\begin{align*}
	0&=  \la vG_2 vf, f\ra=
	\lim_{\lambda\to0} \bigg\la \frac{R_0(\lambda^2)-G_0}{\lambda^2}vf,
	vf\bigg\ra=\lim_{\lambda\to 0} \frac{1}{\lambda^2} \int_{\R^2} \bigg(\frac{1}{4\pi^2\xi^2+\lambda^2}-
	\frac{1}{4\pi^2\xi^2}\bigg) |\widehat{vf}|^2\, d\xi\\
	&=\frac{1}{4\pi^2}\lim_{\lambda\to0} \int_{\R^2} \frac{1}{\xi^2(4\pi^2\xi^2+\lambda^2)} |\widehat{vf}|^2\, d\xi
	=\frac{1}{(2\pi)^4}\int_{\R^2} \frac{ |\widehat{vf}|^2}{\xi^4}\, d\xi.
\end{align*}
Where we used the monotone convergence theorem in the last step.  By the assumptions on $v$ and
$f$, $vf\in L^1$, and hence  $vf=0$. We also know that $f\in S_1L^2$ and hence $f=w\psi$, which implies that $f=0$.
This establishes the invertibility of the operator $S_3 vG_2vS_3$ on $S_3L^2$.

Further, we have the identity for any $f\in S_3L^2$,
\begin{align}\label{G0 to G2 ident}
	\la vG_2 vf, f\ra&=\frac{1}{(2\pi)^4}\int_{\R^2} \frac{ |\widehat{vf}|^2}{\xi^4}\, d\xi=\frac{1}{(2\pi)^4}\bigg\la \frac{\widehat{vf}(\xi)}{\xi^2},
	 \frac{\widehat{vf}(\xi)}{\xi^2}\bigg\ra=\la (-\Delta)^{-1} vf,(-\Delta)^{-1}vf\ra\nn\\
	 &=\la G_0vf,G_0vf\ra.
\end{align}
\end{proof}

\begin{lemma}\label{S3 characterization}
Assume that  $v(x)\les \la x\ra^{-3-}$. Then $\phi=w\psi\in S_3L^2$
if and only if $\psi\in L^2 $ with $H\psi=0$.
\end{lemma}

\begin{proof}
For $\phi\in S_3L^2\subset   S_1L^2$,  we proved above that $\phi=w\psi$, with
$$
H\psi =0,\,\,\,\,\, \psi=c_0-G_0v\phi,\,\,\,\,\,c_0=\frac1{\|V\|_{L^1}}\la v,T\phi \ra.
$$
Also note that $c_0=0 $ by \eqref{ptphi} since $\phi\in S_2L^2$. Therefore, using \eqref{G0 to G2 ident}, we have
\begin{align*}
	\|\psi\|^2_2=\la \psi,\psi\ra=\la G_0v\phi,G_0v\phi\ra=\la v G_2v\phi,\phi\ra < \infty
\end{align*}
by the decay assumption on $v$.

On the other hand if we assume that $\psi \in L^2$ with $H\psi=0$, we   have that $c_0=0$, and hence by Lemma~\ref{S2 characterization}, we have
$\phi=w \psi\in S_2L^2$.  We need to prove that $S_2vG_1vS_2\phi=0.$
Note that, as operators on $L^2$,
$S_2 vG_1vS_2=S_2vWvS_2$, where $W$ is the integral operator with
kernel $-2x\cdot y$. This is because $G_1(x,y)=|x-y|^2=|x|^2-2x.y+|y|^2$, and the contribution of $|x|^2+|y|^2$ is zero since $PS_2=S_2P=0$.

We claim that if $\psi \in L^2$ with $H\psi=0$, then
\begin{align}\label{yvphi}
	\int_{\R^2} yv(y)\phi(y)\, dy=0.
\end{align}
This implies that
\begin{align*}
	S_2 v G_1 v S_2\phi=S_2vWvS_2\phi=-2S_2v(x)x\cdot  \int_{\R^2} yv(y)\phi(y)\, dy=0,
\end{align*}
and hence $\phi\in S_3L^2$.

It remains to prove the claim above. In what follows below we can assume that $|x|>4$ since $\psi\in L^\infty$.
Define the set $B:=\{ y\in \R^2: |y|<|x|/8\}$.
Recall that we have $\psi=-G_0v\phi$, and
as $P\phi=0$ we have
\begin{multline}\label{psiL2}
	\psi(x) =\frac{1}{4\pi} \int_{\R^2} \ln\bigg(\frac{|x-y|^2}{|x|^2}\bigg)v(y)\phi(y)\, dy\\
	 =\frac{1}{4\pi} \int_{B} \ln\bigg(\frac{|x-y|^2}{|x|^2}\bigg)v(y)\phi(y)\, dy
	+\frac{1}{4\pi} \int_{\R^2\setminus B} \ln\bigg(\frac{|x-y|^2}{|x|^2}\bigg)v(y)\phi(y)\, dy.
\end{multline}
First we note that the second term is in $L^2$. Indeed,    using \eqref{logtrick}, and then $1\lesssim \la y\ra/\la x\ra$, we see that
\begin{multline*}
	\Big|\int_{\R^2\setminus B} \ln\bigg(\frac{|x-y|^2}{|x|^2}\bigg)v(y)\phi(y)\, dy\Big|
	 \les \int_{\R^2\setminus B} (1+|y|^{0+}+|x-y|^{0-})|v(y)\phi(y)|\, dy\\
	 \les \frac{1}{\la x\ra^{1+}} \int_{\R^2\setminus B}\la y\ra^{1+}(1+|y|^{0+}+|x-y|^{0-})|v(y)\phi(y)|\, dy \les \frac{1}{\la x\ra^{1+}}\in L^2(\R^2).
\end{multline*}
We now examine the integral on $B$.  We note that on $B$, $\big||y|^2-2x\cdot y\big|/|x|^2<\frac{1}{2}$, and hence
\begin{align*}
	\ln\bigg(\frac{|x-y|^2}{|x|^2}\bigg) =\ln\bigg(1+\frac{|y|^2}{|x|^2}-\frac{2x\cdot y}{|x|^2}\bigg)
%	=\frac{|y|^2}{|x|^2}-\frac{2x\cdot y}{|x|^2}+O\bigg(\bigg(\frac{|y|^2}{|x|^2}-\frac{2x\cdot y}{|x|^2}\bigg)^2\bigg)
=-\frac{2x\cdot y}{|x|^2}+O\bigg( \frac{\la y\ra^{1+}}{\la x\ra^{1+}}\bigg).
\end{align*}
 So that
\begin{align*}
	\frac{1}{4\pi} \int_{B} &\ln\bigg(\frac{|x-y|^2}{|x|^2}\bigg)v(y)\phi(y)\, dy
%\\&
=-\frac{x}{2\pi|x|^2}\cdot \int_{B}  y v(y)\phi(y)\, dy
	+O\bigg( \frac{\int_{B} \la y\ra^{1+} |v(y)\phi(y)|\, dy}{\la x\ra^{1+}} \bigg).
\end{align*}
The error term is in $L^2$.
We also note that
\begin{align*}
	\Big|\int_{\R^2\setminus B} \frac{x\cdot y}{|x|^2}v(y)\phi(y)\, dy \Big| \les
	\int_{\R^2\setminus B} \frac{\la y\ra^{1+}}{\la x\ra^{1+}}|v(y)\phi(y)|\, dy\les \la x\ra^{-1-}\in L^2(\R^2).
\end{align*}
Therefore, we can rewrite the main term as
\begin{align*}
	-\frac{x}{2\pi|x|^2}\cdot \int_{B}  y v(y)\phi(y)\, dy
	=-\frac{x}{2\pi|x|^2}\cdot \int_{\R^2} y v(y)\phi(y)\, dy +O_{L^2}(1).
\end{align*}
Using this in \eqref{psiL2}, we obtain
\begin{align*}
	\psi(x)= \Psi(x)-\frac{x}{2\pi|x|^2} \cdot \int_{\R^2} yv(y)\phi(y)\, dy
\end{align*}
with $\Psi\in L^2$.  As $x/|x|^2$ is not in $L^2(\R^2)$, we must have  \eqref{yvphi}.
\end{proof}

\begin{lemma}\label{eigenproj lem}
The operator $G_0v S_3[S_3vG_2vS_3]^{-1}S_3 vG_0$ is the orthogonal projection on $L^2 $
onto the zero energy eigenspace of $H=-\Delta+V$.
\end{lemma}

\begin{proof}
Let $\{\phi_j\}_{j=1}^N$ be an orthonormal basis for the $S_3L^2$, the range of $S_3$.  Then, we have
\begin{align}\label{psi j eqn}
	\phi_j+wG_0v\phi_j=0, \qquad 1\leq j\leq N.
\end{align}
We have $\phi_j=w\psi_j$ for each $j$ with $\psi_j\in L^2 $. Since $PS_3=0$, we also have
\begin{align*}
	\int_{\R^2} V(x)\psi_j(x)\, dx=\int_{\R^2} v(x)\phi_j(x)\, dx =0.
\end{align*}
Since  $\{\phi_j\}_{j=1}^N$ is linearly independent, we have that $\{\psi_j\}_{j=1}^N$ is linearly independent, and
it follows from \eqref{psi j eqn} that
\begin{align*}
	\psi_j+G_0V\psi_j=0, \quad 1\leq j\leq N.
\end{align*}
Using the orthonormal basis for $S_3 L^2$, we have that for any $f\in L^2(\R^2)$,
$S_3f=\sum_{j=1}^N \la f,\phi_j\ra \phi_j$. Therefore, we have
\begin{align}\label{S2 sum}
	S_3v G_0f=\sum_{j=1}^N \la f,G_0v\phi_j\ra \phi_j=-\sum_{j=1}^N \la f,\psi_j\ra \phi_j.
\end{align}
Let $A=\{A_{ij}\}_{i,j=1}^N$ be the matrix representation of $S_3 vG_2vS_3$ with respect to the
orthonormal basis of $S_3L^2$.   Using \eqref{G0 to G2 ident},
\begin{align*}
	A_{ij}&=\la \phi_i,S_3vG_2vS_3\phi_j\ra=\la G_0v\phi_i,G_0v\phi_j\ra =\la G_0V\psi_i,G_0V\psi_j\ra =\la \psi_i,\psi_j\ra.
\end{align*}
Let $P_e:=G_0vS_3[S_3vG_2vS_3]^{-1}S_3vG_0$.  Then by \eqref{S2 sum}, for any $f\in L^2(\R^2)$,
\begin{align*}
	P_e f&=-\sum_{j=1}^{N}G_0vS_3[S_3vG_2vS_3]^{-1}  \phi_j \la f,\psi_j\ra\\
	&=-\sum_{i,j=1}^N G_0vS_3(A^{-1})_{ij}\phi_i\la f,\psi_j\ra=\sum_{i,j=1}^N \psi_i (A^{-1})_{ij}
	\la f,\psi_j\ra.
\end{align*}
Note that for $f=\psi_k$, $1\leq k\leq N$,
\begin{align*}
	P_e \psi_k=\sum_{i,j=1}^N \psi_i (A^{-1})_{ij} \la \psi_k,\psi_j\ra =\sum_{i,j=1}^N \psi_i
	(A^{-1})_{ij} A_{jk}= \psi_k.
\end{align*}
Thus, we can conclude that the range of $P_e$ is
equal to the span of $\{\psi_j\}_{j=1}^N$ and that $P_e$ is the identity on the range of $P_e$.
Since $P_e$ is self-adjoint, the claim is proven.
\end{proof}

\section{A Weighted Estimate}\label{sec:weighted}
In this section we prove Theorem~\ref{eval only thm}. Recall that if zero is an eigenvalue but there are  neither  s-wave nor p-wave resonances at zero, then $S_1=S_2=S_3\neq 0$.
We note that in this case many terms in the expansions of $M^{\pm}(\lambda)^{-1}$
in Corollaries~\ref{M exp cor1} and ~\ref{M exp cor2} disappear.  This follows as now
\be\label{projec}
PS_1=S_1P=0, \,\,\,\,\,\,S_1TP=PTS_1=0,\,\,\,\,\,\, S_1vG_1vS_1=0.
\ee
We will also need a finer expansion for $M_0(\lambda)$ then it is given in Lemma~\ref{lem:M_exp} to prove the theorem.   Define $g_2^\pm(\lambda)=\lambda^4(a_2\log\lambda+b_{2,\pm})$ and $g_3(\lambda)=a_3\lambda^4$
with $a_2, a_3\in \R\setminus\{0\}$ and $b_{2,-}=\overline{b_{2,+}}$. Also let
$G_3$ be the integral operator with the kernel $|x-y|^4$, and $G_4$ with the kernel
$|x-y|^4\log |x-y|$.
Similar to the expansion given in Lemma~\ref{lem:M_exp} we obtain
\be\label{M0further}
M_0^{\pm}(\lambda)=g_1^{\pm}(\lambda)vG_1v+  \lambda^2 vG_2v
+g_2^{\pm}(\lambda)vG_3v+g_3(\lambda) vG_4v+\mathcal O_1(\lambda^{9/2}),
\ee
by expanding the Bessel functions to order $z^{6}\log z$ and estimating the error term as
in Lemma~\ref{lem:M_exp}. This requires that $|V(x)|\les \la x\ra^{-11-}$.
\begin{prop}\label{prop:Binverseeval}
Assume that $S_1=S_2=S_3$, and that $|V(x)|\les \langle x\rangle^{-11-}$. Then, $B_{\pm}$ is invertible on
$S_1L^2(\R^2)$, and we have
	\begin{multline}\label{B evalonly}
		B^{-1}_{\pm}=\frac{D_3}{\lambda^2}+\frac{g_2^{\pm}(\lambda)}{\lambda^4}D_3\Gamma_2D_3
		+D_3\Gamma_3D_3
		+\frac{g_1^{\pm}(\lambda)^2}{\lambda^4h^{\pm}(\lambda)}D_3\Gamma_4D_3
		+\frac{1}{h^{\pm}(\lambda)}D_3 \Gamma_5D_3\\
		+\frac{g_1^{\pm}(\lambda)}{\lambda^2h^{\pm}(\lambda)}D_3\Gamma_6D_3 +\mathcal O_1(\lambda^{1/2}),
	\end{multline}
where  $\Gamma_i$ are real-valued, absolutely bounded operators on $L^2$.
\end{prop}
\begin{proof}
We will modify the proof of Proposition~\ref{prop:Binverse2}.
Using \eqref{projec} in \eqref{newBdefn} we see that
$B=E(\lambda)$ where (from \eqref{E defn})
\begin{align*}
	E(\lambda)=S_1A^{-1}(\lambda)M_0(\lambda)A^{-1}(\lambda)S_1
-S_1A^{-1}(\lambda)
	[M_0(\lambda)A^{-1}(\lambda)]^2\big[\mathbbm{1}+M_0(\lambda)A^{-1}(\lambda)\big]^{-1}S_1.
\end{align*}
Since $S_1=S_2=S_3$, using \eqref{projec} and \eqref{Ainverse} we have
\be\label{A-1S}
S_1A^{-1}(\lambda)=A^{-1}(\lambda)S_1=S_3.
\ee
Using this, Lemma~\ref{lem:M_exp} and the fact that $A^{-1}(\lambda)=\mathcal O_1(1)$, we obtain
\begin{align}\label{new B expansion}
B&=E(\lambda)=S_3M_0(\lambda)S_3-S_3M_0(\lambda)A^{-1}(\lambda)M_0(\lambda)S_3
+\mathcal O_1(\lambda^{9/2}).
\end{align}
Using \eqref{M0further} and the fact that $S_3vG_1vS_3=0$, we get
$$
S_3 M_0(\lambda)S_3=  \lambda^2 S_3vG_2vS_3
+g_2(\lambda)S_3vG_3vS_3+g_3(\lambda) S_3vG_4vS_3+\mathcal O_1(\lambda^{9/2}).
$$
We now note that by writing $G_1(x,y)=|x|^2-2x\cdot y+|y|^2$,  and using \eqref{yvphi} and $P\perp Q>S_3$, one obtains
\be\label{qg1s}
S_3vG_1vQ=QvG_1vS_3=0.
\ee
Using this and \eqref{projec} in \eqref{Ainverse} and \eqref{M0further}, we have
\begin{align*}
	S_3M_0(\lambda)A^{-1}(\lambda)M_0(\lambda) S_3
	&=\frac{g_1(\lambda)^2}{h(\lambda)}S_3vG_1vPvG_1vS_3
	+\frac{ \lambda^4}{h(\lambda)}S_3vG_2vSvG_2vS_3\\
	&+\frac{ \lambda^2g_1(\lambda)}{h(\lambda)}[S_3vG_1vSvG_2vS_3+
	S_3vG_2vSvG_1vS_3]\\
	&+ \lambda^4S_3vG_2vQD_0QvG_2vS_3+\mathcal O_1(\lambda^{6-}).
\end{align*}
Therefore, using these expansions in \eqref{new B expansion}, we have
\begin{align*}
	B&=\lambda^2 \Gamma_1+g_2^{\pm}(\lambda)\Gamma_2+\lambda^4\Gamma_3
	+\frac{g_1^{\pm}(\lambda)^2}{h^{\pm}(\lambda)}\Gamma_4
	+\frac{\lambda^4}{h^{\pm}(\lambda)} \Gamma_5
	+\frac{\lambda^2g_1^{\pm}(\lambda)}{h^{\pm}(\lambda)}\Gamma_6
	+\mathcal O_1(\lambda^{9/2}),
\end{align*}
where $\Gamma_i$ are absolutely bounded operators on $L^2$ with
$\Gamma_i=S_3\Gamma_i S_3$, and $\Gamma_1^{-1}=D_3$.
%So that
%\begin{align*}
%	B^{-1}=\lambda^{-2}D_3&[\mathbbm 1+\frac{g_2^{\pm}(\lambda)}{\lambda^2}\Gamma_2D_3
%	+\lambda^2\Gamma_3D_3
%	+\frac{g_1^{\pm}(\lambda)^2}{\lambda^2h^{\pm}(\lambda)}\Gamma_4D_3
%	+\frac{\lambda^2}{h^{\pm}(\lambda)} \Gamma_5D_3\\
%	&+\frac{g_1^{\pm}(\lambda)}{h^{\pm}(\lambda)}\Gamma_6D_3
%	+\lambda^2\Gamma_7D_3+\mathcal O_1(\lambda^{5/2})]^{-1}
%\end{align*}
Inverting this via Neumann Series yields the claim of the proposition.
\end{proof}

\begin{corollary}\label{M evalonlycor}

	Assume that $S_1=S_2=S_3$, and that $|V(x)|\les \la x\ra^{-11-}$.  Then
	\begin{multline}
		M^{\pm}(\lambda)^{-1}=\frac{D_3}{\lambda^2}+ (a_1\log\lambda+b_{1,\pm}) \Xi_1+\Big(1+\frac{b_{3,\pm}}{a_2\log\lambda+b_{2,\pm}}\Big)\Xi_2 \\
		+\frac{1}{h^{\pm}(\lambda)}\Xi_3+(M^{\pm}(\lambda)+S_1)^{-1}+ \mathcal O_1(\lambda^{1/2}).
		\label{evalyonly Minv}
	\end{multline}
	Here, $\Xi_i$ are real-valued absolutely bounded operators, $\Xi_2$ and $\Xi_3$ have a projection orthogonal to
	$P$ on at least one side, and $\Xi_1$ have orthogonal projections on both sides. Further $a_i\in\R\setminus\{0\}$ and
	$b_{i,+}=\overline{b_{i,-}}$.

\end{corollary}
We should note that in the statement of the corollary we listed only one term of each form. For example there are several different terms of the form $\frac{b_{3,\pm}}{a_2\log\lambda+b_{2,\pm}}\Xi_2$ in the expansion.
\begin{proof}[Proof of Corollary~\ref{M evalonlycor}]
Using \eqref{M plus S}, \eqref{Ainverse} and \eqref{A-1S}, and then \eqref{projec} and \eqref{qg1s}, we have
\begin{align*}
	&(M(\lambda)+S_1)^{-1}S_3=S_3+A^{-1}(\lambda)M_0(\lambda)S_3+\mathcal O_1(\lambda^{4-})\\
	&=S_3+\frac{g_1(\lambda)}{h(\lambda)}[PvG_1vS_3-QD_0QTPvG_1vS_3]
	+ \lambda^2 QD_0QvG_2vS_3 + \frac{\lambda^2}{h(\lambda)} SvG_2vS_3
	+\mathcal O_1(\lambda^{4-}),\\
	&S_3(M(\lambda)+S_1)^{-1}=S_3+S_3M_0(\lambda)A^{-1}(\lambda)+\mathcal O_1(\lambda^{4-})\\
	&=S_3+\frac{g_1(\lambda)}{h(\lambda)}[S_3vG_1vP-S_3vG_1vPTQD_0Q]
	+ \lambda^2 S_3vG_2vQD_0Q
	+ \frac{\lambda^2}{h(\lambda)} S_3vG_2vS
	+\mathcal O_1(\lambda^{4-}).
\end{align*}
Using these and \eqref{B evalonly} yields that
\begin{multline*}
	(M(\lambda)+S_1)^{-1}S_3 B^{-1}S_3 (M(\lambda)+S_1)^{-1}
%	&=\bigg[S_1+\frac{g_1(\lambda)}{h(\lambda)}E
%	+ \lambda^2 QD_0QvG_2vS_1
%	+ \frac{\lambda^2}{h(\lambda)} SvG_2vS_1
%	+\mathcal O_1(\lambda^{4-})\bigg]\\
%	&\bigg(\frac{D_3}{\lambda^2}+\frac{g_2^{\pm}(\lambda)}{\lambda^4}D_3\Gamma_2D_3
%	+D_3\Gamma_3D_3
%	+\frac{g_1^{\pm}(\lambda)^2}{\lambda^4h^{\pm}(\lambda)}D_3\Gamma_4D_3
%	+\frac{1}{h^{\pm}(\lambda)}D_3 \Gamma_5D_3\\
%	&+\frac{g_1^{\pm}(\lambda)}{\lambda^2h^{\pm}(\lambda)}D_3\Gamma_6D_3
%	+D_3\Gamma_7D_3+\mathcal O_1(\lambda^{1/2})\bigg)\\
%	&\bigg[S_1+\frac{g_1(\lambda)}{h(\lambda)}\widetilde{E}
%	+ \lambda^2 S_1vG_2vQD_0Q + \frac{\lambda^2}{h(\lambda)} S_1vG_2vS
%	+\mathcal O_1(\lambda^{4-})\bigg]\\
	=\frac{D_3}{\lambda^2}+(a_1\log\lambda+b_{1,\pm}) \Xi_1
\\+\Big(1+\frac{b_{3,\pm}}{a_2\log\lambda+b_{2,\pm}}\Big)\Xi_2
	+\frac{1}{h^{\pm}(\lambda)}\Xi_3+\mathcal O_1(\lambda^{1/2}).
\end{multline*}
Applying Lemma~\ref{JNlemma} finishes the proof.
\end{proof}

Using Corollary~\ref{M evalonlycor} and Lemma~\ref{eigenproj lem} in \eqref{resolvent id}, we see that
the contribution of the $D_3/\lambda^2$ term can be written as
\begin{align}
		 &R_0^{+}(\lambda^2) v\frac{D_3}{\lambda^2}v R_0^{+}(\lambda^2)
		 =\frac1{\lambda^2}\big(R_0^{+}(\lambda^2) -g^+(\lambda)\big)  vD_3v
		\big(R_0^{+}(\lambda^2) -g^+(\lambda)\big)\nn\\
		&=\frac1{\lambda^2}\big(R_0^+(\lambda^2) -g^+(\lambda)-G_0 \big)vD_3v\label{Ozero+part}
		\big(R_0^+(\lambda^2) -g^+(\lambda)-G_0 \big)\\
		&+\frac1{\lambda^2}G_0  vD_3v \big(R_0^+(\lambda^2) -g^+(\lambda)
		-G_0 \big)\label{anotherpart} + \frac1{\lambda^2}\big(R_0^+(\lambda^2) -g^+(\lambda)-G_0 \big)vD_3v G_0  \\
		&+ \frac1{\lambda^2} P_e. \label{evalpart}
	\end{align}
In the first line above, we  used the fact that $PD_3=D_3P=0$ to subtract off $g^+(\lambda)$.
\begin{lemma}\label{D3 evalonlylem}

    Under the assumptions of Theorem~\ref{eval only thm},
    if we project away from the eigenspace of $H=-\Delta+V$ at zero energy, for Schwartz functions $f$ and $g$
    the following bound holds.
    \begin{align*}
		\bigg| \int_0^\infty& \frac{e^{it\lambda^2} \chi(\lambda)}{\lambda}
		\big\la[R_0^{+}(\lambda^2) v D_3
		v R_0^{+}(\lambda^2) -R_0^{-}(\lambda^2) v D_3
		vR_0^{-}(\lambda^2)]f,g\big\ra\, d\lambda \bigg| \les |t|^{-1} \|f\|_{L^{1,1+}}\|g\|_{L^{1,1+}}.
    \end{align*}

\end{lemma}

\begin{proof} First note that since we project away the zero eigenspace, the contribution of \eqref{evalpart} cancels out. To bound the contribution of other terms recall that
  the  expansion for $R_0^{\pm}(\lambda^2)$ used in Lemma~\ref{lem:M_exp} gives
	\begin{align*}
	 \big|R_0^+(\lambda^2)(x,y) -g^+(\lambda)-G_0(x,y)\big|&\lesssim \lambda^{1+}|x-y|^{1+}\lesssim \lambda^{1+}(\la x\ra^{1+}+\la y\ra^{1+}),\\
	\big|\frac{\partial}{\partial\lambda}\big(R_0^+(\lambda^2)(x,y) -g^+(\lambda)-G_0(x,y)\big)\big|&\lesssim \lambda^{0+}(\la x\ra^{1+}+\la y\ra^{1+}).
	\end{align*}
Note that if $|\Phi(\lambda)|+\lambda|\Phi^\prime(\lambda)|\lesssim \lambda^{1+}$ and same for $\Psi$, then
\begin{align}
		\bigg|\int_0^\infty \frac{e^{it\lambda^2} \chi(\lambda)}{\lambda} \Phi(\lambda)\Psi(\lambda)\, d\lambda\bigg|
		&\les |t|^{-1} \int_0^\infty \bigg|\frac{d}{d\lambda}\bigg(\frac{\chi(\lambda) \Phi(\lambda)\Psi(\lambda)}{\lambda^2}\bigg)\bigg|\, d\lambda
		\les |t|^{-1}.
	\end{align}
Also using that $|D_3|:L^2\to L^2$ and $\la x\ra^{1+}v(x)\in L^2$, the contribution of \eqref{Ozero+part} satisfies the claim of the lemma.
	
	For the contribution of the terms in \eqref{anotherpart}, we need to use the cancellation between the `+' and `-' terms in Stone's formula.
	\begin{multline*}
		G_0  vD_3v \big([R_0^+(\lambda^2) -g^+(\lambda)]
		-[R_0^-(\lambda^2) -g^-(\lambda)]
		\big)\\
		=G_0 vD_3v \big(2i J_0(\lambda|\cdot|)-2i \Im(z)\big)
		=2iG_0  vD_3v  J_0(\lambda|\cdot|).
	\end{multline*}
	Where we used \eqref{Q trick} in the last step.  As in
	the case of an s-wave resonance, we separate into the high and low energies. For the low energy part we use \eqref{Q trick} and investigate
	\begin{align}\label{low123}
		\int_0^\infty e^{it\lambda^2} \frac{\chi(\lambda)}{\lambda}\bigg( \chi(\lambda|y-y_1|)J_0(\lambda|y-y_1|)
		-\chi(\lambda(1+|y_1|))J_0(\lambda(1+|y_1|))\bigg)\, d\lambda.
	\end{align}
	After an integration by parts, the result relies on proving the following bound.
	\begin{align*}
		\int_0^1 \bigg|\frac{d}{d\lambda}\bigg(\frac{\chi(\lambda|y-y_1|)J_0(\lambda|y-y_1|)
		-\chi(\lambda(1+|y_1|)J_0(\lambda(1+|y_1|))}{\lambda^2}\bigg)\bigg|\, d\lambda < C_{y,y_1}.
	\end{align*}
	We need not consider when the derivative acts on the cut-off function as this restricts us to an annulus where
	$\lambda\sim |y-y_1|^{-1}$ or $\lambda \sim (1+|y_1|)^{-1}$ and we can bound \eqref{low123} by
	\begin{align*}
	  \int_{\lambda \sim |y-y_1|^{-1} }\frac{1}{\lambda^2}\, d\lambda \les |y-y_1|\les \la y\ra \la y_1 \ra.
	\end{align*}
	The analogous bound holds for $\lambda\sim (1+|y_1|)^{-1}$.
	With this in mind, it suffices to prove that
	\begin{align*}
	  \bigg|\int_0^1 \frac{f(\lambda |y-y_1|)-f(\lambda (1+|y_1|))}{\lambda^3}\, d\lambda \bigg|\les \la y\ra^{2+} \la y_1\ra^{1+}
	\end{align*}
	with $f(z):=\chi(z)[-2J_0(z)+zJ_0^\prime(z)]$.  As we are restricted to low energy, we have the expansion for
	$J_0(z)$ and its derivative in a powers of $z$ from \eqref{J0 def}.  So that
	\begin{align*}
	    f(z)=a_0+a_1 z+a_2 z^2+ O(z^{2+}).
	\end{align*}
	Since $f(0)=-2J_0(0)=-2$ we have that $a_0=-2$. Further $f'(0)=-J_0^\prime (0)+zJ_0''(z)\big|_{z=0}=0$, and $f''(z)=zJ_0'''(z)$. Therefore, we have $a_1=a_2=0$, and
	\begin{align}\nn
	    f(z)=-2+O(z^{2+}), \qquad f'(z)=O(z^{1+}).
	\end{align}
	Now,
	\begin{multline*}
	     \big|f(\lambda|y-y_1|)-f(\lambda(1+|y_1|))\big|=\Big|\int_{\lambda(1+|y_1|)}^{\lambda|y-y_1|} f'(z)\, dz\Big|\\
	     \lesssim \lambda^{2+} \big| |y-y_1|-(1+|y_1| \big|  \big( |y-y_1|^{1+}+\la y_1\ra^{1+}\big) \lesssim \lambda^{2+} \la y\ra^{2+} \la y_1 \ra^{1+}.
	\end{multline*}
	Thus, we have
	\begin{align*}
	    \eqref{low123}\les \Big|\int_0^1 \frac{f(\lambda |y-y_1|)-f(\lambda (1+|y_1|))}{\lambda^3}\, d\lambda\Big|
	    \les \la y\ra^{2+} \la y_1\ra^{1+} \int_0^1 \lambda^{-1+}\, d\lambda \les \la y\ra^{2+} \la y_1\ra^{1+},
	\end{align*}
	as desired.

	Now, for the high energy we proceed along the lines of the proof of Lemma~\ref{low high log YJ lemma}.  We employ the function $\widetilde{G}^{\pm}(\lambda, |y-y_1|,1+|y_1|)$ of Lemma~\ref{G2 lemma}.  Specifically, we need to
	bound
	\begin{align*}
	    \int_0^\infty e^{it\lambda^2}\chi(\lambda) G_0 \frac{vD_3v}{\lambda} \widetilde{G}^{\pm}(\lambda,p,q)\, d\lambda
	\end{align*}
	with $p=\max(|y-y_1|,1+|y_1|)$ and $q=\min(|y-y_1|,1+|y_1|)$.  We  will
	apply Lemma~\ref{stat phase} to
	\begin{align*}
	    a(\lambda)=\frac{\chi(\lambda)\widetilde{G}^{\pm}(\lambda)}{\lambda}.
	\end{align*}
	Using the bounds of Lemma~\ref{G2 lemma} with $\tau=1$, we have
	\begin{align*}
	    &|a(\lambda)|\les \la y\ra \bigg(\frac{\widetilde{\chi}(\lambda p)}{|\lambda p|^{1/2}}
	    +\frac{\widetilde{\chi}(\lambda q)}{|\lambda q|^{1/2}}\bigg)
	    \les \la y\ra \big(\widetilde{\chi}(\lambda p)\sqrt{p\lambda}
	    +\widetilde{\chi}(\lambda q) \sqrt{q\lambda}\big),\\
	    &|a^\prime(\lambda)|\les \frac{1}{\lambda}
	    \la y\ra \bigg(\frac{\widetilde{\chi}(\lambda p)}{|\lambda p|^{1/2}}
	    +\frac{\widetilde{\chi}(\lambda q)}{|\lambda q|^{1/2}}\bigg)
	    \les \la y\ra \bigg(\frac{p\widetilde{\chi}(\lambda p)}{|\lambda p|^{1/2}}
	    +\frac{q\widetilde{\chi}(\lambda q)}{|\lambda q|^{1/2}}\bigg).
	\end{align*}
	 Here we used that on the support of $\widetilde{\chi}(\lambda p)$, we have $1\les \lambda p$.  At this point,
	 the proof follows exactly along the lines of Lemma~\ref{low high log YJ lemma} with the extra weights
	 of $p+q\les \la y\ra \la y_1\ra$, which yields the required bound.
\end{proof}

We are now ready to prove the theorem.  We provide a sketch, as there is a significant overlap with the proofs of
previous estimates in Section~\ref{sec:swave}.

\begin{proof}[Proof of Theorem~\ref{eval only thm}]
We already proved the theorem for the contribution of the $D_3/\lambda^2$ term in
Corollary~\ref{M evalonlycor} to \eqref{resolvent id}. The contribution of the $\Xi_1$ term and the terms in the second line of \eqref{evalyonly Minv} satisfies the dispersive bound by the results of Section~\ref{sec:swave}.
It remains to control the contribution of
	\begin{align*}
\Big(1+\frac{b_{3,\pm}}{a_2\log\lambda+b_{2,\pm}}\Big)\Xi_2.
	\end{align*}
We will only provide a brief sketch. Recall that $\Xi_2$  has projection  orthogonal to $P$ only on one side, say on the right.
On high energy, we can use
$\lambda |x-x_1|\gtrsim 1$ to extract positive powers of $\lambda$ for the integration at the loss of
a weight as in the proof of Lemma~\ref{D3 evalonlylem}.
The polynomial weights arising are either ameliorated by the decay of the potential $v$ or goes into the weight of
the weighted dispersive bound. For the low energy part, the worst case is when we have $Y_0$ on both sides. This arises only with the term containing $\log\lambda$ in the denominator due to the cancellation between the $\pm$ terms.
On the right hand side, using \eqref{Q trick}, we replace $\chi Y_0$ with $F$ from Lemma~\ref{lem:FG} to
reduce to bounding the following integral
$$
 \bigg|\int_0^\infty   e^{it\lambda^2}\lambda \chi(\lambda)  \chi(\lambda|x-x_1|)
    Y_0(\lambda|x-x_1|) \frac{b_{3 }}{a_2\log\lambda+b_{2 }}
    F(\lambda,y,y_1)\, d\lambda\bigg|.
$$
After an integration by parts it suffices to prove that
\begin{multline*}
 \int_0^\infty   \bigg| \frac{d}{d\lambda}\bigg\{ \chi(\lambda)  \chi(\lambda|x-x_1|)
    Y_0(\lambda|x-x_1|) \frac{b_{3 }}{a_2\log\lambda+b_{2 }}
    F(\lambda,y,y_1)\,\bigg\}\bigg| d\lambda \\ \les k(x,x_1)k(y,y_1)\la x \ra^{1+}\la y \ra^{1+}\la x_1\ra^{1+}\la y_1\ra^{1+}.
\end{multline*}
We note, from \eqref{Y0 def}, that
\begin{align*}
    \chi(\lambda|x-x_1|)Y_0(\lambda |x-x_1|)=\frac{2}{\pi}[\log \lambda+\log |x-x_1|+\gamma]+ O(\lambda^{1+}|x-x_1|^{1+}).
\end{align*}
The first $\log\lambda$ is the most troubling, we note that to control it we use the following
facts
\begin{align*}
    \bigg| \log(\lambda) \frac{b_3}{a_2\log\lambda+b_2}\bigg|\les 1,\,\,\,\,\,\,\,\,\,\,
     \bigg|\frac{d}{d\lambda}\bigg(\log(\lambda) \frac{b_3}{a_2\log\lambda+b_2}\bigg)\bigg|
    \les \frac{1}{\lambda (\log\lambda)^2}.
\end{align*}
The contribution of the other terms can be bounded by similar arguments.
\end{proof}

\end{document}